\documentclass[hidelinks,onefignum,onetabnum]{siamart220329}

\usepackage{amsfonts}
\usepackage{amssymb}
\usepackage{pdfpages}
\floatstyle{plain}
\restylefloat{figure}
\usepackage{cleveref}
\usepackage[normalem]{ulem}

\usepackage{color}

\newcommand{\OUT}[1]{}

\usepackage{bbm}
\newcommand{\ONE}{\mathbbm{1}}

\newcommand{\pr}{\mbox{\sf P}}
\newcommand{\ex}{\mathbb{E}}               

\newcommand{\cov}{\mbox{\sf Cov}}

\newcommand{\lip}{\mbox{\sf Lip}}

\newcommand{\Real}{\mathbb R}
\newcommand{\Natural}{\mathbb N}
\newcommand{\Complex}{\mathbb C}

\newcommand{\bZ}{{\mathbb Z}}
\newcommand{\bX}{{\bf X}}               

\newcommand{\mcmcX}{{X}}

\newcommand{\al}{\alpha}                
\newcommand{\Lam}{\Lambda}               

\newcommand{\ra}{\rightarrow}           

\usepackage[linesnumbered,ruled,algo2e]{algorithm2e}

\newtheorem{remark}{Remark}

\title{Regenerative Ulam-von Neumann Algorithm: An Innovative Markov chain Monte Carlo Method for Matrix Inversion}
\ifpdf
\hypersetup{
  pdftitle={Regenerative Ulam-von Neumann Algorithm: A Markov chain Monte Carlo Method for Matrix Inversion},
  pdfauthor={S. Ghosh, L. Horesh, V. Kalantzis, Y. Lu, T. Nowicki}
}
\fi

\author{Soumyadip Ghosh$^*$, Lior Horesh$^*$,  Vassilis Kalantzis$^*$, Yingdong Lu$^*$, and Tomasz  Nowicki\thanks{IBM Research, Thomas J. Watson Research Center, Yorktown Heights, NY 10598. 
  \email{\{ghoshs,lhoresh,vkal,yingdong,tnowicki\}@us.ibm.com}.}}

\begin{document}

\maketitle

\begin{abstract}
This paper presents a regenerative variant of the classical Ulam-von Neumann Markov chain 
Monte Carlo algorithm for the approximation of the matrix inverse. The algorithm presented 
in this paper, termed \emph{regenerative Ulam-von Neumann algorithm}, utilizes the 
regenerative structure of classical, non-truncated Neumann series defined by a non-singular 
matrix and produces an estimator of the matrix inverse via ratios of 
unbiased estimators of the regenerative quantities. The accuracy of the proposed algorithm depends on a single 
parameter that controls the total number of simulated Markov transitions, thus avoiding the challenge of balancing between the total number of Markov chain replications and their length as in the classical Ulam-von Neumann algorithm. To efficiently utilize Markov chain transition samples in the calculation of the regenerative variables, the proposed algorithm automatically quantifies the contribution of each Markov transition to all regenerative quantities by a carefully designed updating scheme that utilized three separate matrices containing the current weights, total weights, and regenerative cycle count, respectively. A probabilistic analysis of the performance of the algorithm, including the variance of the estimator, is provided. Finally, numerical experiments verify the effectiveness of the proposed scheme.
\end{abstract}



\section{Introduction}
\label{sec:intro}

Markov chain Monte Carlo (MCMC) is an important class of algorithms that draw sample paths 
from a Markov chain \cite{brooks1998markov}. One well-known application of MCMC is the numerical solution of linear systems, also known as the Ulam-von Neumann algorithm due to its original application in thermonuclear dynamics computations by Stanis\l{}aw Ulam and 
John von Neumann~\cite{ForsytheLeibler1950}. The main idea behind the Ulam-von Neumann 
algorithm lies in constructing and sampling a discrete Markov chain with a state space 
defined by the rows of the iteration matrix and a carefully designed transition probability matrix \cite{benzi2017analysis,doi:10.1137/130904867}. Monte Carlo-type methods can be 
appealing on a variety of applications due to their high parallel granularity and  ability 
to compute partial solutions of systems of linear equations \cite{wu2019multiway}.

In this paper we consider the related problem of computing an approximation of the 
inverse of a non-singular matrix via MCMC  \cite{strassburg2013monte,DIMOV200125,meyer1975role,heyman1995accurate,dimov1998new,prabhu2013approximative}. A straightforward implementation of the Ulam-von Neumann 
algorithm samples an associated Markov chain to estimate a truncation of the Neumann series $(I-A)^{-1}=\sum_{n=0}^\infty A^n$. More specifically, let $P$ be the transition matrix for the Markov chain defined over the row indices of the matrix $A$. Then, the $(i,j)$-th element of the $n$-th power of $P$ represents the probability that the Markov chain transits from the state $i$ to state $j$ after $n$ steps. This probability can be estimated by Monte Carlo simulations of the Markov chain, and thus the inverse of a matrix expressed via a converging  Neumann series in~\cite{ALEXANDROV1998113,doi:10.1080/10637199408962545,doi:10.1137/130904867,DIMOV200125}. The accuracy of classical MCMC algorithms depends on the total number of walks of the Markov chain as well as the truncation threshold (length) of each independent walk, and their interaction, e.g., 
see~\cite{9652800}. 

The algorithm presented in this paper, termed as \emph{regenerative Ulam-von Neumann algorithm}, utilizes the regenerative structure of the Neumann series and produces unbiased estimators of the key regenerative weights, i.e., we estimate the non-truncated Neumann series. A probabilistic analysis of the regenerative algorithm is also conducted where the second moment of the samples is quantified through two independent central limit theorems; one for Markov chains on symmetric groups and one for the summation of weakly dependent random variables. 

A list of the main contributions of this paper is summarized as follows: 
\begin{itemize}
\item 
{\bf Algorithm:} Upon observing a regenerative structure in the classical Ulam-von Neumann algorithm, we identify a stochastic relationship~\eqref{eqn:gen_defn} that is induced by the regenerative structure. This observation allows us to introduce an alternative estimator in the form of regenerative quantities which are also solutions to the stochastic fixed-point equation. This new estimator, listed as Algorithm~\ref{alg:ruvn}, consists of unbiased estimators of key quantities in the matrix inversion, in contrast to the classic Ulam-von Neumann estimator that suffers from the bias induced by truncating the Neumann series. 
\item 
{\bf Practicality:}
The accuracy of the proposed algorithm depends only on a single parameter --the total number of Markov transitions simulated-- thus avoiding the challenge of balancing between two different parameters as in the classical Ulam-von Neumann matrix inversion algorithm. To efficiently utilize the Markov chain transition samples in the calculation of the regenerative quantities, the proposed algorithm (Algorithm~\ref{alg:ruvn}) quantifies the contribution of each Markov transition to all regenerative quantities via exploiting three separate matrices that maintain quantities needed to calculate the current sample of $A_{i,j}$, the cumulative sum of all $A_{i,j}$ samples observed, and the count of regenerative cycles, respectively. An immediate benefit of the regenerative approach is the update of all entries of the $j$-th column of the matrix inverse simultaneously per each drawn sample $A_{i,j}$. Thus, if one is interested in computing only one or a few columns of the matrix inverse, e.g., computation of partial correlations in data uncertainty 
quantification \cite{kalantzis2018scalable,kenett2015partial}, the proposed algorithm can reduce the overall computational cost when the matrix $A$ is implicitly-defined and sampling $A_{i,j}$ is computationally intensive.
\item
{\bf Analysis:} To quantify the variance of the estimator, a central limit theorem (CLT) described in Theorem~\ref{thm:main_CLT} is established. There are two key components in proving Theorem~\ref{thm:main_CLT}. First, we present a connection between the evolution of the walk counting matrix and a Markov chain on symmetric groups, and obtain a CLT in Theorem~\ref{thm:Gamma_CLT}. Second, we obtain a CLT in Theorem~\ref{thm:Sigma_CLT} for the cumulative sum matrix, as a direct result of a general Lindberg CLT for weakly dependent random vectors (Theorem~\ref{thm:main_clt}) and whose proof follows basic approaches in the literature but offers more precise estimations. 
\end{itemize}

\textbf{Related work.} Approximate inverse computations are often encountered in matrix computations, e.g., preconditioning by (sparse) approximate inverse matrices \cite{benzi1999comparative}. Common deterministic techniques to approximate the matrix inverse rely either on applying Newton's method to the function $f(X)=X^{-1}-A$ (Newton-Schulz method \cite{higham2008functions}; also referred to as Hotelling's method \cite{hotelling1943some}) or 
optimization, e.g., by minimizing $\|XA-I\|_F$ \cite{benzi1996sparse}. In this paper we focus on MCMC  approaches to compute an approximate inverse. MCMC methods offer a probabilistic alternative to classical matrix inversion, with roots in the Neumann series and random walk theory \cite{ForsytheLeibler1950,malioutov2006walk}. 
The work in \cite{gower2017randomized} presents a family of stochastic/randomized algorithms for calculating an approximate inverse matrix, with block variants of quasi-Newton updates as a special case; see also \cite{gower2015randomized} for similar work from the viewpoint of a randomized linear system solver, and \cite{wu2019multiway,okten2005solving,magalhaes2022distributed,strassburg2013monte,benzi2002parallel,benzi2017analysis} for various Monte Carlo algorithms suggested for the solution of linear systems. An analysis of the convergence of Monte Carlo algorithms from the viewpoint of the transition probability matrix can be found in \cite{doi:10.1137/130904867}. 
Several MCMC algorithms for the problem of matrix inversion are compared in \cite{alexandrov1996comparison} while a reinforcement learning approach is presented in \cite{barto1993monte}. 
Finally, MCMC algorithms have been developed for graph centrality computations of large-scale graphs \cite{guidotti2024fast,magalhaes2022distributed}.

\textbf{Organization.} Section~\ref{sec:bac} summarizes related previous work, and in 
particular the Ulam-von Neumann algorithm for matrix inversion. Section~\ref{sec:uvn} 
presents our main algorithmic contribution that exploits the regenerative structure of 
Ulam-von Neumann. In Section~\ref{sec:analysis}, we present a probabilistic analysis of the 
proposed algorithm. Applications and
numerical experiments are presented in Section~\ref{sec:applications}. Finally, 
our concluding remarks are presented in Section~\ref{sec:conclusions}. 

\textbf{Notation.} The $ij$-th entry of the matrix $A$ is denoted by $A_{i,j}$, while the 
$ij$-th entry of the matrix power $A^k,\ k\in \mathbb{Z}$, is denoted by $\left[A^k\right]_{i,j}$. 
The operation $\lfloor x \rfloor$ will round $x$ down to its nearest integer less than or 
equal to it. The term $\mathbb{E}[\cdot]$ denotes the expectation operator, and we will 
denote the expectation of a random variable $\Delta$ by $\mathbb{E}[\Delta]$. The spectral 
radius of the matrix $A$ is equal to the maximum of the absolute values of its eigenvalues 
and will be denoted by $\rho(A)$. The symbol ${\bf 1}_d$ denotes the $d\times 1$ vector of 
all ones while on certain occasions we write  $[d]$ to represent the set $\{1,2\ldots,d\}$. 
The symbol $\exp[x]$ is equivalent to $e^x$ and will be used to reduce 
cluttering of notation whenever necessary. Finally, the variable $\ONE_{[\al=\beta]}$ is 
equal to one when $\al$ is equal to $\beta$, and zero otherwise.

\section{Markov chain Monte Carlo (MCMC) matrix inversion (Ulam-von Neumann algorithm)}
\label{sec:bac}

Consider a $d\times d$ matrix $B=I-A$ where $\rho(A) <1$. The matrix $B$ is non-singular 
and its inverse matrix\footnote{To simplify notation, our presentation focuses on $B^{-1}$ 
but approximating $A^{-1}$ can be achieved in an identical manner by running MCMC on 
$I-\gamma A$ and post-multiplying with an appropriate $\gamma \in \mathbb{R}$.} can be represented via the Neumann series 
\begin{equation*}
B^{-1}=\sum_{k=0}^\infty A^k.
\end{equation*}
Denote $C=B^{-1}=(I-A)^{-1}$. Let $\mcmcX(k)$ be a \emph{periodic and irreducible} Markov chain defined over state 
space $[d]$ and associated with the $d\times d$ transition probability matrix $P$~\cite{Asmu03}, 
where $P_{i,j}$ denotes the transition probability from state $i\in [d]$ to state $j\in [d]$.
 The entries of the matrix $P$ are set such that the sum of the entries of its rows is equal to one \cite{benzi2017analysis}, and $P_{i,j}>0$ if and only if $A_{i,j}\neq 0$. Specifically, we set $P_{i,j}=\frac{|A_{i,j}|}{\sum_{k=1}^d |A_{i,k}|}$ as this simplifies the ratios $A_{i,j}/P_{i,j}$ (see, e.g.,~\cite{ALEXANDROV1998113}) that form a key element in our calculations below.

In the following, we describe the Ulam-von Neumann algorithm to compute $C_{i,j}$ for any index pair $(i,j)\in[d]\times [d]$. 
Let the $r$-th replicate, $r\in[R],\ R\in \mathbb{N}$, of the Markov chain $\mcmcX(k)$ 
be represented as $\mcmcX_r(k)$ and initiated at some state $i\in [d]$, i.e., $\mcmcX_r(0)=i$. Furthermore, let $W^r_0=1$ and define the following sequence of (scalar) random variables for $k\in \mathbb{N}$:
\begin{align*}
W^r_k=W^r_{k-1}\frac{A_{\mcmcX_r(k-1), \mcmcX_r(k)}}{P_{\mcmcX_r(k-1), \mcmcX_r(k)}}.
\end{align*}
Note that $A$ can have negative entries, and so we refer to the ratio in $W^r_k$ generally 
as the \emph{weight} of the transition. The key idea of the Ulam-von Neumann algorithm is then based on the following simple observation.

\begin{lemma}
\label{lem:expectation}
For any initial state $i\in[d]$, 
\begin{equation*}
\ex\left[W_k^r\ONE_{[\mcmcX_r(k)=j]}\right]=\left[A^k\right]_{i,j},\end{equation*}
i.e., the expectation of the random variable $W_k^r\ONE_{[\mcmcX_r(k)=j]}$ is equal to the $(i,j)$-th element of $A^k$
for any $r=1,2,\ldots, R$, and $k\in \mathbb{N}$.
\end{lemma}
\begin{proof}
This can be proved by induction. The case $k=0$ and $k=1$ are straightforward. 
For the general case where we update $W_{k-1}^r$ to $W_{k}^r$, we can write:
\begin{align*}
\ex\left[W_{k}^r\ONE_{[\mcmcX_r(k)=j]}\right] =&\ex[\ex[W_{k}^r\ONE_{[\mcmcX_r(k)=j]}|\mcmcX_r(k-1)]]\\=&\ex\left[\ex\left[W_{k-1}^r\frac{A_{\mcmcX_r(k-1), \mcmcX_r(k)}}{P_{\mcmcX_r(k-1), \mcmcX_r(k)}}\ONE_{[\mcmcX_r(k)=j]}\bigg|\mcmcX_r(k-1)\right]\right]
\\=&\ex\left[\left[A^{k-1}\right]_{i,\mcmcX_r(k-1)}\frac{A_{\mcmcX_r(k-1), j}}{P_{\mcmcX_r(k-1), j}}\right]\\=& \left[A^{k}\right]_{i,j}.
\end{align*}
\end{proof}

Let now $\mcmcX_r(k)$ terminate after exactly $k_r\in\mathbb{N}$ transitions, all starting 
from the same state $i\in [d]$, where $k_r$ is constant for all $r=1,2,\ldots,R$. 
The classical Ulam-Von Neumann algorithm approximates $C_{i,j}$ by the formula
\begin{align}
\label{eqn:output}
\widehat{C}^R_{i,j}:=\frac{1}{R} \sum_{r=1}^R\left[\sum_{k=0}^{k_r}W^{r}_k\ONE_{[\mcmcX_r(k)=j]}\right].
\end{align}
Lemma~\ref{lem:expectation} establishes that $W_k^r\ONE_{[\mcmcX_r(k)=j]}$ is an unbiased estimator of $\left[A^k\right]_{i,j}$. Therefore, as $k_r$ tends to infinity, the average 
of the sum $\sum_{k=0}^{k_r}W^r_k\ONE_{[\mcmcX_r(k)=j]}$ approaches $C_{i,j}$, and thus $\widehat{C}^R_{i,j}$ can be viewed as an estimator of $C_{i,j}$. 
A necessary and sufficient condition for the convergence of the Ulam-von Neumann algorithm is as follows. 

\begin{theorem}[Theorem 4.2 in~\cite{doi:10.1137/130904867}]\label{thm1}
Given a $d\times d$ transition probability matrix $P$ and a non-singular matrix 
$A$ such that $\rho(A)<1$, the Ulam-von Neumann algorithm converges if and 
only if $\rho\left(\widehat{H}\right)<1$, where 
$
    \widehat{H}_{i,j} =
    \begin{cases}
      \left[A^2\right]_{i,j}/P_{i,j}        & \mathrm{if}\ P_{i,j}\neq 0 \\
      0        & \mathrm{otherwise}.
    \end{cases}.
$
\end{theorem}
The algorithm and analysis presented throughout the rest of this paper assume that 
the transition probability matrix $P$ satisfies the convergence criterion in Theorem 
\ref{thm1}.

\section{A regenerative Ulam-von Neumann algorithm: random walks and stochastic fixed-point equations} \label{sec:uvn}

In this section, we present a new MCMC matrix inversion algorithm based on the regenerative 
structure of the discrete Markov chain $\mcmcX(k)$. In contrast to classical MCMC, the 
regenerative approach can approximate all entries of the matrix inverse while starting from 
any randomly initialized state. Moreover, only one such initialization is required since for 
any state $i\in [d]$ the Markov chain $\mcmcX(k)$ restarts itself every time the same state 
is observed, and the non-overlapping segments of the Markov chain are independent.

\subsection{The basics of the regenerative algorithm}
\label{sec:alg_basic}

Let us focus first on the approximation of a diagonal entry $C_{ii},\ i\in [d]$. Instead of $R$ separate Markov chains, we now consider a single Markov chain $\mcmcX(k)$ that starts at state $i$. We let the Markov chain progress 
for an indeterminate number of steps while registering the portions of the 
chain that occur between returns to the initial state $i$. Since each such 
portion occurs independently of the other, the state $i$ is a regeneration 
point \cite{smith1955regenerative}. For any $k\in \mathbb{N}\cup\{0\}$ we define 
the initial state variable $\tau^{ii}_0=0$ followed by 
\begin{align*}
    \tau^{ii}_{k+1}&=\inf\{t>\tau^{ii}_k | \mcmcX(t)=i\}.
\end{align*}
The variable $\tau^{ii}_{k+1}$ denotes the $(k+1)$-th return to state $i$ 
and the Markov chain does not visit state $i$ between $\tau^{ii}_k$ and $\tau^{ii}_{k+1}$. 
Following the above definition of $\tau^{ii}_{k+1}$, we can write  $W_{\tau^{ii}_{k+1}}=W_{\tau^{ii}_{k}} \cdot \al^{ii}_k$, where by the regenerativity 
of the Markov chain $\mcmcX(k)$, the scalars $\al^{ii}_k$ are i.i.d. random variables 
counting the increment of $W_k$ between the $\tau^{ii}_k$ and $\tau^{k}_{k+1}$. More 
precisely, for any $k\ge 1$, we can write
\begin{align*}
\al^{ii}_k=\prod_{\ell=\tau^{ii}_{k-1}+1}^{\tau^{ii}_{k}}\frac{A_{\mcmcX(\ell-1),\mcmcX(\ell)}}{P_{\mcmcX(\ell-1),\mcmcX(\ell)}}.
\end{align*}
Thus, setting $W_0=1$, we can rewrite 
\begin{align}
\label{eqn:gen_defn}
Z_{ii}:=\sum_{k=0}^{\infty}W_k\ONE_{[\mcmcX(k)=i]}&=\sum_{n=0}^\infty W_{\tau_n}\\
=\sum_{n=0}^\infty W_0\prod_{v=1}^n \al^{ii}_v 
\quad&=\quad W_0 + \al^{ii}_1\;\;\sum_{n=1}^{\infty} W_0\prod_{v=2}^{n+1} \al^{ii}_v. \nonumber
\end{align}
Defining $Z_{ii}^{\ge k} = \sum_{n=0}^\infty W_0 \prod_{v=k+1}^{n+k} \al^{ii}_v$ as the infinite sum gathered after the $k$-th return to state $i$, the Markovian memoryless-ness property implies that each $Z_{ii}^{\ge k}$  is identically distributed to $Z_{ii}$. Moreover, $\al^{ii}_1$ and $Z_{ii}^{\ge 1}$ are independent by the same property since they are separated by $X_{\tau^{ii}_1}$ when the chain passes through $i$. Thus, taking expectations on both sides and rearranging:
\begin{equation*}
    \ex[Z_{ii}]=\dfrac{1}{1-\ex[\al^{ii}]}.
\end{equation*}
Meanwhile, from Lemma~\ref{lem:expectation}, we know that
\begin{align}
\label{eqn:gen_mean}
\ex[Z_{ii}]=\sum_{k=0}^{\infty}\ex\left[W_k\ONE_{[\mcmcX(k)=i]}\right]=
\sum_{k=0}^{\infty}\left[A^k\right]_{ii}=C_{i,i}.
\end{align}

For the off-diagonal terms $Z_{ij}, i\neq j$, once again consider 
the starting point $i\in [d]$ of the Markov chain $\mcmcX(k)$ and 
define 
$$\tau^{ij}=\inf\{t>0, |X(t)=j\}\, \, \tau_0^{ij}=0;\quad \al^{ij}=\prod_{\ell=1}^{\tau^{ij}}\frac{A_{\mcmcX(\ell-1),\mcmcX(\ell)}}{P_{\mcmcX(\ell-1),\mcmcX(\ell)}}.$$ Then, we have
\begin{align*}
Z_{ij}  & :=\sum_{k=0}^{\infty}W_k\ONE_{[\mcmcX(k)=j]}\stackrel{d}{=}\sum_{n=0}^\infty\al^{ij} \prod_{k=1}^n\al_k^{jj}=\al^{ij} Z_{jj}.
\end{align*}
Moreover, $\al^{ij}$ and $Z_{jj}$ are independent in the last term. Hence, $C_{i,j}=\ex[\al^{ij}]C_{j,j}$. 

\subsection{Implementation of the regenerative algorithm}

The key to implementing the regenerative algorithm is to estimate the expectations of the cumulative cycle weights $\alpha^{ij}$ for $i,j\in[d]$. 
To keep track of the running weights, cumulative weights and counting of the regenerative cycles (i.e., the disjoint segments of the chain), we introduce three $d\times d$ matrices, $\Lam$, $\Sigma$, and $\Gamma$, respectively.
Once the chain is initialized at some random state $i\in [d]$, the algorithm iterates 
until  $\min_{q,l} \Gamma_{q,l}\ge N$ where $N\in \mathbb{N}$ controls the variance 
of the estimator; see Theorem~\ref{thm:main_CLT}. The proposed regenerative Ulam-von 
Neumann Algorithm is summarized in Algorithm \ref{alg:ruvn}.

\begin{algorithm}
    \setcounter{AlgoLine}{0}
    \SetKwInOut{Input}{Input}
    \SetKwInOut{Output}{Output}

    \Input{$A\in \mathbb{R}^{d\times d},\ P\in \mathbb{R}^{d\times d}$, and scalar $N\in \mathbb{N}$. Also define and set the matrices 
    $\Lam\in \mathbb{R}^{d\times d},\ \Sigma\in \mathbb{R}^{d\times d},\ \Gamma\in \mathbb{N}^{d\times d}$ equal to zero.    }
    {\bf Initialization:} Initialize the Markov chain at some random state $i\in [d]$. \\
    \While{$\min_{q,l} \Gamma_{q,l}< N$}{
  \ (a) \emph{Update the $i$-th row of $\Lam$}:
    $\Lam_{i,k}=\Lam_{i,k}+\ONE_{[\Lam_{i,k}=0]}$, for all $k=1,2, \ldots, d$;
    \\
    (b) \emph{Sample next state $j\in [d]$ of Markov Chain using probability vector $P_{i,\cdot}$}
    \\
    (c) \emph{Update $\Lam$:}
    $\Lam=\Lam\times \frac{A_{i,j}}{P_{i,j}}$;
    \\
    (d) \emph{Update $j$-th column of $\Gamma$:} 
    $\Gamma_{k,j}=\Gamma_{k,j}+1$ if $\Lam_{k,j}\neq 0$, for all $k=1,2, \ldots, d$;
    \\
    (e) \emph{Update $j$-th column of $\Sigma$:} 
    $\Sigma_{k,j}=\Sigma_{k,j}+ \Lam_{k,j}$, for all $k=1,2, \ldots, d$;
    \\
    (f) \emph{Reset the $j$-th column of $\Lam$: }
    $\Lam_{k,j}=0$ for all $k=1,2, \ldots, d$; 
    \\
    (g) \emph{Over-write $i\leftarrow j$.}
    }
    \Output{\begin{align*}
\widehat{C}_{i,j} =\left\{ \begin{array}{cc} \dfrac{1}{1-\dfrac{\Sigma_{i,i}}{\Gamma_{i,i}}} & \mathrm{if}\ \ i=j\\\\ \dfrac{\Sigma_{i,j}}{\Gamma_{i,j}} \dfrac{1}{1-\dfrac{\Sigma_{j,j}}{\Gamma_{j,j}}} & \mathrm{if}\ \  i\neq j.
\end{array} \right.
\end{align*}}
 \caption{Regenerative Ulam-von Neumann Algorithm}\label{alg:ruvn}
\end{algorithm}

The running weights matrix $\Lam$ keeps track of the sample of weight $\alpha^{ij}$ during the current regenerative cycle, and the cumulative weights matrix $\Sigma$ maintains their summation over all the observed cycles. The cycle count matrix $\Gamma$ then gives us the necessary information to construct sample average estimators for all the cycle weights. When the chain reaches state $i\in [d]$, all cycles starting at state $i$ are re-initiated and the elements of the $i$-th row of $\Lam$ which are equal to zero (i.e., corresponding to previous cycles that were closed) are reset to one. As the chain arrives at the next state 
$j\in [d]$, the entries of the matrix $\Lam$ are scaled by ${A_{i,j}}/{P_{i,j}}$ 
and all cycles that terminate at state $j$ close, while the information 
from the $j$-th column of the matrix $\Lam$ is transferred (added) to the respective column of the cumulative weights $\Sigma$. Finally, the $j$-th column of the matrix $\Lam$ is reset to zero. In order to keep the count of the number of cycles in the chain, the $j$-th column of the matrix $\Gamma$ increases by one. The loop continues until the count of the performed cycles is large enough, i.e., larger than an input parameter $N$. Note that the reset of the $j$-th column of $\Lam$ to zero at the end of the current cycle followed by the update of the $i$-th row by one at the start of the next cycle assures that the segments are disjoint and therefore independent.

\begin{remark}
The choice of the matrix $P$ can have a significant impact on the variance of the estimator via influencing the accumulated ratios in $\al^{ii}$. While the optimal choice of MC transition 
matrices to minimize variance is an extensively studied subject, e.g., see~\cite{Glynn2016ExactSV}, 
this topic is beyond the scope of the current article and is left as future work. 
\end{remark}

\color{black}

\subsection{A note on the computation of individual entries}

In contrast to the classical Ulam-von Neumann matrix inversion algorithm, Algorithm 
\ref{alg:ruvn} updates several entries of the matrix inverse per access of a matrix 
entry $A_{i,j}$; see also \cite{guidotti2024fast}. Updating multiple entries is  
appealing when the matrix $A$ is 
implicitly-defined and only a subset of entries of the matrix inverse is sought, 
since pre-computing and storing explicitly all non-zero entries of the matrix 
$A$ is then impractical. Appendix B discusses a variant of 
Algorithm \ref{alg:ruvn} that computes a single column of $B^{-1}$.

\section{Probabilistic analysis of the regenerative algorithm}
\label{sec:analysis}

The analysis in Section \ref{sec:alg_basic} established that the regenerative estimator listed 
in Algorithm \ref{alg:ruvn} is unbiased. To complete the analysis, we need to establish its
second-order statistical properties with a CLT-type of result. In~\cite{10.1093/biomet/asz002} 
it is demonstrated that for multivariate Markov chain Monte Carlo methods, establishing a 
multivariate Markov chain central limit theorem for the underlying process is key to answering 
when the sampling process should terminate so that the desired accuracy can be achieved. 
Thus, our goal is to establish a CLT so that our framework and results can be readily applied.

Let $K \in \mathbb{N}$ denote the number of Markov chain 
transitions (iterations) that have been sampled thus far by Algorithm 
\ref{alg:ruvn}, and let $\Sigma^K$ and $\Gamma^K$ denote the corresponding 
matrices generated by Algorithm \ref{alg:ruvn}. For any pair of states 
$i\in[d],\ j\in[d]$, the scalar $\Sigma_{i,j}^{K}$ produced in Steps 2(d) 
and 2(e) in Algorithm~\ref{alg:ruvn} accumulates the weights over $\Gamma^K_{i,j}$ 
transitions (or cycles) where the Markov chain starts at state $i\in[d]$ and 
transits into state $j\in[d]$ for the first time.

\begin{theorem}
\label{thm:Gamma_CLT}
For an invertible $d\times d$ matrix $A$ and a Markov chain on the state space of the rows of $A$ with a transition matrix $P$ satisfying the convergence criterion in Theorem \ref{thm1}, there exist a $d\times d$ matrix $\Gamma_s$ and a symmetric and positive semi-definite $d^2\times d^2$ matrix $M_s$, such that the quantity $\sqrt{K}\left(\frac{\Gamma^K}{K}-\Gamma_s\right)$ converges to a $d^2$-variate normal variable with covariance matrix $M_s$ as $K\to \infty$.
\end{theorem}
\begin{proof}
The proof is presented in Section~\ref{sec:gamma_matrix}.
\end{proof}
\begin{theorem}
\label{thm:Sigma_CLT}
For an invertible $d\times d$ matrix $A$ and a Markov chain on the state space of the rows of $A$ with a transition matrix $P$ satisfying the convergence criterion in Theorem \ref{thm1}, there exist a $d\times d$ matrix $\Gamma_s$ and a $d^2\times d^2$ symmetric and positive semi-definite matrix $M_a$, such that the quantity $\sqrt{K}\left(\frac{\Sigma^K}{K}-\widehat{\Gamma}_s\right)$, with $\left[\widehat{\Gamma}_s\right]_{i,j}=[{\Gamma}_s]_{i,j}C_{i,j}$ for $i,j\in [d]$,  converges to a $d^2$-variate normal variable with covariance matrix $M_a$ as $K\to \infty$.
\end{theorem}
\begin{proof}
The proof is presented in Section~\ref{sec:sigma_matrix}.
\end{proof}

For the sake of convenience, in the following theoretical result we re-write 
the output obtained after $K$ transitions of Algorithm~\ref{alg:ruvn} as the $d\times d$ matrix $\widehat{C}^K$ with entries  
\begin{align*}
\widehat{C}^K_{i,j} = \frac{\Sigma^K_{i,j}}{\Gamma^K_{i,j}}D_j^K +\ONE_{i=j},\;\;\text{where}\;\;D_j^K:= \frac{\Gamma^K_{j,j}}{\Gamma^K_{j,j}-\Sigma_{j,j}^K}.
\end{align*}
The above definition of $\widehat{C}^K_{i,j}$ is identical to that in the output of Algorithm~\ref{alg:ruvn}, except that in the latter case we divide both the 
numerator and denominator by $\Gamma^K_{j,j}$. 

\begin{theorem}
\label{thm:main_CLT} 
For an invertible $d\times d$ matrix $A$ and a Markov chain on the state space of the rows of $A$ with a transition matrix $P$ satisfying the convergence criterion in Theorem \ref{thm1}, 
there exists a $d\times d$ symmetric and positive semi-definite matrix $M_C$, such that the matrix with entries $\left\lbrace\sqrt{K}\left(\widehat{C}_{i,j}^K-C_{i,j}\right)\right\rbrace_{1\le i,j\le d}$ converges to a $d^2$-variate normal 
variable with covariance matrix $M_C$ as $K\ra \infty$. 
\end{theorem}
\begin{proof}
Let $\gamma^K_{ij}= \ex[\Gamma^K_{i,j}]$ denote the expected number of 
$(i,j)$ regeneration cycles after $K$ iterations and denote by $\overline{\Sigma_{i,j}^{K}}$ 
the summation of weights after $\gamma^K_{ij}$ cycles. When $\Gamma^K_{i,j}\ge \gamma^K_{i,j}$, 
the variable $\overline{\Sigma_{i,j}^{K}}$ is obtained by truncating the weight cumulation 
beyond $\gamma^K_{i,j}$ regeneration cycles, otherwise, it is obtained by adding weights to 
the $(i,j)$ cycles generated according to the same distribution. Then, we can re-write 
$\sqrt{K}\left(\widehat{C}_{i,j}^K-C_{i,j}\right)$ as follows: 
\begin{align*}
&\sqrt{K}\left(\frac{\Sigma_{i,j}^K}{\Gamma^K_{i,j}}D_j^K-(C_{i,j}-\ONE_{i=j})\right)
\\
=  &\sqrt{K}\left(\frac{\Sigma_{i,j}^K}{\sqrt{\gamma^K_{i,j}}}D_j^K -\frac{\Gamma^K_{i,j}}{\sqrt{\gamma^K_{i,j}}}(C_{i,j}-\ONE_{i=j})\right)\frac{\sqrt{\gamma^K_{i,j}}}{\Gamma^K_{i,j}}\\
= & \frac{\sqrt{K\gamma^K_{i,j}}}{\Gamma^K_{i,j}}\left(\frac{\overline{\Sigma_{i,j}^{K}}}{\sqrt{\gamma^K_{i,j}}}D_j^K +\frac{\Sigma_{i,j}^K-\overline{\Sigma_{i,j}^{K}}}{\sqrt{\gamma^K_{i,j}}}D_j^K -\sqrt{\gamma^K_{i,j}}(C_{i,j}-\ONE_{i=j})\right.
\\ & \left.-\frac{\Gamma^K_{i,j}-\gamma^K_{i,j}}{\sqrt{\gamma^K_{i,j}}}(C_{i,j}-\ONE_{i=j})\right)
\\ 
=& \frac{\sqrt{K\gamma^K_{i,j}}}{\Gamma^K_{i,j}}\left(\frac{\overline{\Sigma_{i,j}^{K}}}{\sqrt{\gamma^K_{i,j}}}D_j^K -\sqrt{\gamma^K_{i,j}}(C_{i,j}-\ONE_{i=j})\right) + \frac{\sqrt{K{\gamma^K_{i,j}}}}{\Gamma^K_{i,j}}\left(\frac{\Sigma_{i,j}^K-\overline{\Sigma_{i,j}^{K}}}{\sqrt{{\gamma^K_{i,j}}}}\right)D_j^K
\\&
-\sqrt{K}\frac{\Gamma^K_{i,j}-{\gamma^K_{i,j}}}{{\Gamma}^K_{i,j}}(C_{i,j}-\ONE_{i=j}).
\end{align*}

Thus, the matrix with individual entries $\left\lbrace\sqrt{K}\left(\widehat{C}_{i,j}^K-C_{i,j}\right)\right\rbrace_{1\le i,j\le d}$ can be written as the sum of three matrices 
whose corresponding $(i,j)$ entries are determined by the first, second, and third term 
of the addition listed in the above summation. With the aid of Slutsky's theorem (also known as “converging together'' Theorem, e.g., ~\cite{grimmett2001probability}), the first and third matrices converge to $d^2$-variate normal distributions due to Theorems~\ref{thm:Sigma_CLT} and~\ref{thm:Gamma_CLT}, respectively. 
For the second matrix, following similar arguments as in~\cite[Theorem 3.1]{gut2009stopped} and~\cite{10.1145/1225275.1225279}, we can show that $\left(\Sigma_{i,j}^K-\overline{\Sigma_{i,j}^{K}}\right)/\sqrt{{ \gamma}^K_{i,j}}\ra 0$, in probability, as $K\ra \infty$. To see this, for any $\epsilon, \delta>0$, let $n_1=\lfloor\sqrt{{\gamma}^K_{i,j}} (1-\epsilon)\rfloor$, $n_2=\lfloor\sqrt{{ \gamma}^K_{i,j}} (1+\epsilon)\rfloor +1$. Then, for sufficiently large $K$, we have 
\begin{align*}
&\pr\left[\frac{\left|\Sigma_{i,j}^K-\overline{\Sigma_{i,j}^{K}}\right|}{\sqrt{{ \gamma}^K_{i,j}}}>\epsilon\right]
\le 
\pr\left[\frac{\left|\Sigma_{i,j}^K-\overline{\Sigma_{i,j}^{K}}\right|}{\sqrt{{ \gamma}^K_{i,j}}}>\epsilon,\ \Gamma^K_{i,j}\in [n_1, n_2]\right]+ \pr[\Gamma^K_{i,j}\notin [n_1, n_2]]<\delta,
\end{align*}
where the last inequality follows by noticing that for any $\Gamma^K_{i,j}\in [n_1, n_2]$, $\left(\Sigma_{i,j}^K-\overline{\Sigma_{i,j}^{K}}\right)$ consists of at most $2\epsilon \sqrt{{\gamma}^K_{i,j}}$ terms of summation of regenerative weights. Because of 
Theorem~\ref{thm:Sigma_CLT}, for large enough $K$, we then have  
\begin{align*}
\pr\left[\frac{\left|\Sigma_{i,j}^K-\Sigma_{i,j}^{\bar K}\right|}{\sqrt{{ \gamma}^K_{i,j}}}>\delta/2,\ \Gamma^K_{i,j}\in [n_1, n_2]\right]<\frac{\delta}{2}.
\end{align*}
Meanwhile, Theorem~\ref{thm:Gamma_CLT} leads to 
$\pr\left[\Gamma^K_{i,j}\notin [n_1, n_2]\right]<\frac{\delta}{2}$ 
for sufficiently large $K$. 
\end{proof}

\subsection{CLT-Based Output Analysis}
\label{sec:output_analysis}

The CLT presented in Theorem~\ref{thm:main_CLT} provides a rate at which the estimation error diminishes to zero with respect to the value of $K$. As is further pointed out in ~\cite{10.1093/biomet/asz002}, CLT quantitatively characterizes the quality of the estimator in the form of \emph{confidence domains}, a natural generalization of the widely 
used concept of confidence intervals for univariate estimators. In particular, confidence domains are constructed in~\cite{10.1093/biomet/asz002} as domains that are prescribed 
around the sample value of the estimator and -with high probability- contain the value being estimated. More specifically, consider the estimation of a $p$-dimensional quantity 
$w^*\in \Real^p$ by the vector $w_K$ stemming from $K$ steps of an MCMC estimator, and 
let $M_K\in\Real^{p\times p}$ be a positive semi-definite matrix. For any $\al\in (0,1)$, define the region
\begin{align*}
\mathcal{C}_\al(K):=\{w\in \Real^p : K (w-w_K)^\top M_K^{-1}  (w-w_K)\le T^2_{1-\al, p, q} \},
\end{align*}
where $T^2_{1-\al, p, q}$ is the $(1-\al)$-th quantile of a Hotelling’s $T$-squared 
distribution with dimension $p$, and $q$ degrees of freedom determined by the selection of 
matrix $M_K$. In our case $p=d\times d$, and 
$w_K=[\widehat{C}_{1,1}^K\ \dots\  \widehat{C}_{d,1}^K\ \dots\ \widehat{C}_{1,2}^K\ \dots\ \widehat{C}_{d,2}^K\ \dots\ \widehat{C}_{1,d}^K\ \dots\ \widehat{C}_{d,d}^K]^T$. For the matrix 
$M_K$, one common choice is to choose the sample covariance matrix, which, provided the terminology used in this paper, can be computed as $[M_K]_{i,j}=\frac{1}{K-1}\sum_{k=1}^K([w_{k}]_i-[{\bar w}]_i)([w_{k}]_j-[{\bar w}]_j)$, where ${\bar w}= \frac{1}{K}\sum_{k=1}^K w_k$. Theorem 1 in~\cite{10.1093/biomet/asz002} can be then paraphrased as follows.

\begin{theorem}[{Theorem} 1 in~\cite{10.1093/biomet/asz002}]\label{thmm1}
Let $w_K$ be an unbiased estimator of $w^*$ such that $\lim_{K\ra \infty}\|w_K\|=\|w^*\|$, and assume a CLT such that $\sqrt{K} [w_K-w^*]$ converges in distribution to a $p$-variate normal distribution with mean zero and covariance matrix $M$ such that 
$M_K \ra M$ with probability one. Then, for any $\al\in(0,1)$, the probability $\pr[w^*\in \mathcal{C}_\al(K)]$ converges to $(1-\al)$ as $K\ra \infty$.
\end{theorem}

It is shown, e.g., in~\cite{kook2024covarianceestimationusingmarkov}, that the sample 
covariance matrix $M_K$ indeed converges to the limiting covariance matrix $M$. 
Alternative numerical procedures to estimate the limiting covariance matrix $M$ directly 
can be found in~\cite{agrawal2024markovchainvarianceestimation} and~\cite{TREVEZAS20092242}. Theorem \ref{thmm1} indicates that for a sufficiently large $K$, the region 
$\mathcal{C}_\al(K)$ will contain the desired value $w^\ast$ with a confidence guarantee 
close to the nominal value $(1-\al)$. 
The convergence properties of the regenerative estimator can also be described as a function of $\Gamma_{i,j}^K$ by recalling that the products $\alpha^{i,j}$ collected from each cycle, which are then accumulated in the sums $\Sigma^K_{i,j}$, are i.i.d. Variance estimators for regenerative Markov Chains is an extensively studied subject, e.g., see~\cite{Asmu07Glynn}.

\subsection{Central Limit Theorem~\ref{thm:Gamma_CLT} of matrix $\Gamma$}
\label{sec:gamma_matrix}

In this section we focus on the growth of the entries of the matrix $\Gamma$ in 
Algorithm~\ref{alg:ruvn}. For convenience, we list the computation of the matrix $\Gamma$ separately in Algorithm~\ref{alg:gammamatrix}, which, for the same pseudo-random number sequence, produces the same matrix $\Gamma$ as Algorithm~\ref{alg:ruvn}. Each $d\times d$ matrix $\Pi$ produced by Algorithm~\ref{alg:gammamatrix} is binary but as we discuss next, the output space that the 
matrix $\Pi$ admits can be significantly less than the $2^{d^2}$ possible outcomes. 

\begin{algorithm}
    \setcounter{AlgoLine}{0}
    \SetKwInOut{Input}{Input}
    \SetKwInOut{Output}{Output}

    \Input{$P\in \mathbb{R}^{d\times d}$, scalar $N\in \mathbb{N}$, $d\times d$ matrices  $\Pi=\Gamma=0$.}
    {\bf Initialization:} Initialize the Markov chain at some random state $i\in [d]$. \\
    \While{$\min_{q,l} \Gamma_{q,l}< N$}{\ (a) \emph{Update the $i$-th row of $\Pi$}:
    $\Pi_{i,k}=1$, for all $k=1,\ldots,d$\\
    (b) \emph{Sample next state $j$ of Markov Chain using row probability vector $P_{i,\cdot}$}
    \\
    (c) {\emph{Update $j$-th column of $\Gamma$}: $\Gamma_{k,j}= \Gamma_{k,j}+ \Pi_{k, j}$}, for all $k=1,2, \ldots, d$;\\
    (d) {\emph{Reset the $j$-th column of $\Pi$}: }
$\Pi_{k, j}=0$, for all $k=1,2, \ldots, d$. 
\\
(e) \emph{Over-write $i\leftarrow j$.}
}
    
    \Output{$\Gamma$}
    
    \caption{Calculation of The Matrix $\Gamma$ }\label{alg:gammamatrix}
\end{algorithm}

In each iteration of Algorithm~\ref{alg:gammamatrix}, the matrix $\Pi$ is updated 
twice, i.e., once in step 2(d) where we reset the $j$-th column of $\Pi$, and once again in step 2(a) in the next iteration, where we update the $i$-th row of $\Pi$; note though that due to step 2(e), $i\leftarrow j$ and thus both steps leverage the same newly sampled state $j \in [d]$. We can 
then formalize an iteration of Algorithm~\ref{alg:gammamatrix} via the following operator,  
denoted as $T_j$: 
\begin{align*}
(T_j\Pi)_{u,v}=\left\{\begin{array}{cc} 0 & v=j,\ u\neq j,\\ 1 & u=j, \\ \Pi_{u,v} & \hbox{otherwise.}\end{array}. \right.
\end{align*}
Applying $T_j$ affects only the entries located along the $j$-th row and $j$-th column of the matrix $\Pi$. 
After an initial phase where $T_j$ is invoked at least once for each $j\in[d]$, the entries of the matrix $\Pi$ satisfy the invariant relations $\Pi_{u,v}+\Pi_{v,u}=1$ for all $u\neq v$ and $\Pi_{u,u}=1$, $u\in[d]$. We term this phase as \emph{warm-up phase}. By the \emph{irreducibility and aperiodic assumption} of the Markov chain $X(t)$, we know that this warm-up phase will be finished almost surely. Once the warm-up phase is over, the matrix $\Pi$ satisfies the following lemma. 

\begin{lemma}
\label{lem:repetition}
Let $T^p_J:=T_{j_p}\circ \ldots \circ T_{j_2}\circ T_{j_1}$ with $J:=(j_p, \ldots, j_1)\in[d]^p$, where $\circ$ denotes the composition of operators. Then, for any $i\in[d]$, we have 
$T_i\circ T^p_J \circ T_i\,\Pi=T_i\circ T^p_J\,\Pi$.
\end{lemma}
\begin{proof}
It can be easily seen that $\Pi_{i,j}=1$ and $\Pi_{j,i}=0$, $j\neq i$, for both $T_i\circ T^p_J \circ T_i\,\Pi$ and $T_i\circ T^p_J\,\Pi$. All other values are determined solely by $\Pi$ and $T^p_J$.  
\end{proof}
Lemma~\ref{lem:repetition} implies that for a sequence of $T_j$ operations, repetitions can be simplified to the last applied operator. More precisely, for $T^m_{N_m}=T_{n_m} \circ\dots\circ T_{n_1} $, with $N_m=(n_m, \ldots, n_1)\in[d]^m$, let the set of distinct $n_i$'s have a cardinality $m^*\le \min(m,d)$ and the sequence\footnote{The ordered sequence $N^*$ can be seen naturally as a part of a permutation on $[d]$.}
$N^*=(n^*_{m^*},\dots,n^*_1)$ consist of all distinct $n_j$'s in an order determined by their first appearance from left to right in $N_m$. Moreover, $n^*_{m^*}=n_m$ and $n^*_k=n_\ell$, $\ell=\inf\left\{k: n_k \notin \{n^*_{m^*}, \ldots, n^*_{k+1}\}\right\}$. If $T^{m^*}_{N^*}=T_{n^*_{m^*}}\circ\ldots \circ T_{n_1^*}$, we have
\begin{align*}
T^m_{N_m}\Pi&=T^{m^*}_{N^*}\Pi
\\
T_{n_m} \circ\dots\circ T_{n_1}\Pi&= T_{n_{m^*}^*}\circ\ldots \circ T_{n_1^*}\Pi\,.
\end{align*}
\begin{lemma}
For any permutation $\sigma$ of $[d]$, let
$
T^d_\sigma=T_{\sigma(d)}\circ T_{\sigma(d-1)}\circ\cdots T_{\sigma(2)}\circ T_{\sigma(1)}
$. 
Then, for any pair of binary matrices $\Pi$ and $\Pi'$, we have
$
T^d_\sigma\Pi=T^d_\sigma\Pi'\,.
$
\end{lemma}
\begin{proof}
The proof consists of a verification that the end result of either side does not depend on the initial matrix. 
\end{proof}

Therefore, after the conclusion of the warm-up phase, the matrix $\Pi$ is solely determined by the permutation of $[d]$ corresponding to the last $d$ distinct $T_j$ applied, and thus is independent of the initial state in Algorithm~\ref{alg:gammamatrix}. Because of this, it suffices to initialize $\Pi$ to zero at the beginning of Algorithm~\ref{alg:gammamatrix}. Let us now denote by $S_d$ and $P_d$ the set of all permutations of $[d]$ and the set of all matrices obtained after the warm-up phase, respectively. Then, the application of the last $d$ distinct $T_j$ defines a bijection $K_d$ from $S_d$ to $P_d$. Hence, there exist $d!$ possible distinct matrices in $P_d$.

Consider now the state $X(t)=i$ of the Markov chain $X$ at time $t$, and let $\sigma_t$ be the 
permutation of $[d]$ corresponding to the matrix $\Pi$ generated up to time $t$ by Algorithm \ref{alg:gammamatrix}, i.e., $\sigma_t=K_d^{-1}(\Pi)$. Noticing that $X(t+1)=j$ with probability $P_{i,j}$, 
we have an induced map $\widehat{T}$ on $S_d$ such that 
the updated permutation is determined as $\sigma_{t+1}=\widehat{T}(\sigma_t)= ({\widehat T}_j(\sigma_t),j)$, where for a generic permutation $\sigma=(\sigma(1), \ldots \sigma(\ell-1), r, \sigma(\ell+1),\ldots \sigma(d))$ we define ${\widehat T}_r(\sigma)$ as 
\begin{align*}
{\widehat T}_r(\sigma(1), \ldots \sigma(\ell-1), r, \sigma(\ell+1),\ldots \sigma(d)) = (\sigma(1), \ldots \sigma(\ell-1), \sigma(\ell+1),\ldots \sigma(d)),
\end{align*}
i.e., ${\widehat T}_r(\sigma)$ removes index $r$ from the position $\ell \in [d]$ of the permutation $\sigma$. 
This induces a Markov chain $Y(t)$ over the symmetric group $S_d$ such that $Y(t)=\sigma_t$. The Markov chain $Y(t)$ can be interpreted as a \emph{Random-to-Bottom Shuffling} Markov chain in the jargon of random walk on finite groups (e.g., see~\cite{Saloff-Coste2004}).\footnote{To see the reason for this term, consider a deck of $n$ indexed cards, where at each time an index $j$ is determined with probability $P_{i,j}$ where $i$ denotes the index of the card at the bottom before the card with index $j$ is itself moved to the bottom. The permutations produced by this procedure coincide with those produced by $Y(t)$.} For Markov chains with finite state space, such as $X(t)$ and $Y(t)$, the property of irreducibility implies that any state can be reached in finite time regardless of the present state, and aperiodic chain allows this time to take any value~\cite{Asmu03}. Therefore, it is easy to verify the following lemma.
\begin{lemma}
\label{lem:same_ergodicity}
If $X(t)$ is an irreducible and aperiodic Markov chain with state space $[d]$, $Y(t)$ is an irreducible and aperiodic Markov chain with state space $S_d$.
\end{lemma}

We are now ready to proceed with the proof of Theorem \ref{thm:Gamma_CLT}.

\begin{proof}[Proof of Theorem~\ref{thm:Gamma_CLT}] 
As we demonstrated in this section, after the warm-up phase is completed, the increments of 
the $\Gamma$ matrix correspond to the random selection of $d!$ matrices, according to 
the state in which the sampling Markov chain lies.  
From Lemma~\ref{lem:same_ergodicity}, we know that the Markov chain $Y(t)$ on the symmetric group $S_d$ is geometrically ergodic, which means that the Markov chain converges to its invariant measure exponentially fast, since it is an irreducible and aperiodic Markov chain on a finite state space. Meanwhile, the main CLT from~\cite{ChanGeyer1994} states that for any 
geometrically ergodic Markov chain $Z_n$ defined on a general state space $\bX$, and any Borel function on $\bX$ to a finite-dimensional Euclidean space, $f:\bX\ra \Real^*$, satisfying $\ex_\pi\|f\|^{2+\delta} :=\int_\bX \|f(x)\|^{2+\delta} \pi(dx)<\infty$ for some $\delta>0$, the quantity $\sqrt{n} ({\bar f}_n-\ex_\pi f)$ with ${\bar f}_n:=n^{-1}\sum_{i=1}^nf(Z_i)$ and $\ex_\pi f:=\int_\bX f(x)\pi(dx)$ converges to a Gaussian distribution with a covariance matrix $\Sigma_f$ as $n\ra \infty$. The proof follows from applying this CLT to $Y(t)$ and a uniformly bounded function $f:S_d\ra \{0,1\}^{d\times d}, \sigma\mapsto K_d(\sigma)$. 
\end{proof}

\subsection{Central Limit Theorem~\ref{thm:Sigma_CLT} of matrix $\Sigma$}
\label{sec:sigma_matrix} 

In this section we provide the necessary theoretical discussion to prove Theorem~\ref{thm:Sigma_CLT}. Due to the complexity of the proof, we structure 
our presentation such that Theorem~\ref{thm:Sigma_CLT} is a direct 
consequence of Theorem~\ref{thm:main_clt} and Lemma~\ref{lem:exponential_decay_correlation}. Central limit theorems for dependent random vectors 
have been long studied\footnote{We could not find the exact results required for our purpose in the published literature; references ~\cite{DoukhanWintenberger2006} and~\cite{refId0} are probably the closest related. Therefore, we include a detailed proof of Theorem~\ref{thm:main_clt} for completeness.} in the probability community, e.g., see~\cite{DOUKHAN1999313},~\cite{10.3150/23-BEJ1614}, and~\cite{10.1214/16-AOS1512}. 
Although we follow the general idea of Lindeberg method and block techniques which can be found in the proofs of many CLT, e.g., \cite{DoukhanWintenberger2006}, technical calculations critical to the vector and dependence structures presented in this section, e.g., the usage of Lemma~\ref{lem:delta_moment}, are different from what appears in the current literature. 
The following definition of weak dependence can be found in ~\cite{DOUKHAN1999313} and~\cite{DoukhanWintenberger2006}.

\begin{definition}
\label{def:weak_dependence}
Process $(X_n)_{n\in \bZ}$ is said to be $({\epsilon_r}, \psi)$-weakly dependent, ${r\in\bZ}$, for a sequence of real numbers $\epsilon_r \ra 0$ and a function $\psi: \mathbb{N}^2\times (\Real_+)^2 \ra \Real_+$ if for any $r\ge 0$ and $(u+v)$-tuples of integers $s_1\le \ldots \le s_u\le s_u+r\le t_1\le \ldots t_v$, we have,
\begin{align*}
|\cov(f(X_{s_1},\ldots, X_{s_u}), g(X_{t_1},\ldots, X_{t_v}))| \le \psi(u,v, \lip f, 
\lip g)\epsilon_r
\end{align*}
where $f$ and $g$ are two real valued function satisfying $\|f\|_\infty, \|g\|_\infty \le 1$, and $\lip f$ and $\lip g$ are the Lipschitz constants of functions $f$ and $g$, respectively.
More precisely, we have, 
\begin{align*}
\lip f:= \sup_{(x_1, x_2, \ldots, x_u)\neq (y_1, y_2, \ldots, y_u)} \frac{|f(x_1, x_2, \ldots, x_u)- f(y_1, y_2, \ldots, y_u)|}{\|x_1-y_1\|+\ldots+ \|x_u-y_u\|}.
\end{align*}
\end{definition}
\begin{theorem}
\label{thm:main_clt}
Assume $\Real^d$ stationary random vectors $(X_n)_{n\ge 0}$, $\ex[X_0]=0$ and  $\ex\left[\|X_0\|^m\right] <\infty$ for some $m>2$. Furthermore, if $\{X_n\}$ is a $({\epsilon_r}, \psi)$-weakly dependent sequence with $\epsilon_r =O(r^{-\kappa})$ for $\kappa > (2m-3)/(m-2)$, then ${\widehat X}_n=\frac{1}{\sqrt{n}} \sum_{k=1}^n X_k$ converges in distribution to $N(0, M^2)$ with $M^2= \sum_{k=0}^\infty \cov(X_0, X_k)=\sum_{k=0}^\infty \ex[X_0X_k^T]$. 
\end{theorem}
\begin{proof}
    The Proof can be found in Appendix A. 
\end{proof}

Recall now the Markov chain $X(t)$, and for any pair $i\neq j$ of starting state $i\in[d]$ 
and ending state $j\in[d]$, define the 
variable $\zeta_0^{ij}=\inf\{t\ge 0 |X(t)=i\}$. Moreover, for any $K\ge 1$ define 
\begin{align*}
\tau_{K}^{ij}=&\inf\{t>\zeta_{K-1}^{ij} |X(t)=j\},
\\
\zeta_{K}^{ij}= &\inf\{t>\tau_{K}^{ij} |X(t)=i\},
\\
\al_K^{ij}=&\prod_{\ell=\zeta_{K-1}^{ij}+1}^{\tau_K^{ij}}\frac{A_{\mcmcX(\ell-1),\mcmcX(\ell)}}{P_{\mcmcX(\ell-1),\mcmcX(\ell)}}.
\end{align*}
\begin{lemma}
\label{lem:exponential_decay_correlation}
For any fixed integer $n$ and any $(i,j)\neq (k,\ell)$, the correlation between $\al^{ij}_n$ and $\al^{k\ell}_{n+m}$ decays exponentially as $m$ increases. 
More precisely, there exist constants $C_1, C_2>0$ depending only on the transition probability, such that,  $\cov(\al^{ij}_1, \al^{k\ell}_{m})\le C_1 e^{-C_2 m}$. 
\end{lemma}
\begin{proof}
It suffices to show that the correlations between $\al^{ij}_1$ and $\al^{k\ell}_{m}$ decay exponentially with respect to $m$. To see this, we have, 
\begin{align*}
\cov(\al^{ij}_1, \al^{k\ell}_{m}) = & \ex[\al^{ij}_1 \al^{k\ell}_{m}]-\ex[\al^{ij}_1]\ex[\al^{k\ell}_{m}]\\
= & \ex[(\al^{ij}_1-\ex[\al^{ij}_1]) (\al^{k\ell}_{m}-\ex[\al^{k\ell}_{m}])]
\\= & \ex[(\al^{ij}_1-\ex[\al^{ij}_1])\ONE \{\tau^{ij}_1\ge m\} (\al^{k\ell}_{m}-\ex[\al^{k\ell}_{m}])]
\\ \le &[\ex[(\al^{ij}_1-\ex[\al^{ij}_1])^2\ONE \{\tau^{ij}_1\ge m\}]^\frac12 [\ex(\al^{k\ell}_{m}-\ex[\al^{k\ell}_{m}])^2]^\frac12.
\end{align*}
Since the variable $\al$ is uniformly bounded, there exist a $C>0$, such that 
\begin{align*}
\cov(\al^{ij}_1, \al^{k\ell}_{m}) 
\le &C\pr[\tau^{ij}_1\ge m]^\frac12.
\end{align*}
It is known that hitting time $\tau^{ij}_1$ has finite exponential moments in an ergodic Markov chain, see e.g.~\cite{meyn93}, that is, there exists a $\xi>0$, such that $\ex[\exp[\xi\tau^{ij}_1]]<\infty$, for all $1\le i,j\le d$. Then, Markov's inequality gives us $\pr[\tau^{ij}_1\ge m]=\pr[\exp[\xi\tau^{ij}_1]\ge \exp[\xi m]]\le \ex[\exp[\xi\tau^{ij}_1]]e^{-\xi m}$. Thus, the lemma follows.
\end{proof}
Lemma \ref{lem:exponential_decay_correlation} shows that the correlation between $\al^{ij}_n$ and $\al^{k\ell}_{n+m}$ decays exponentially with respect to the difference of the cycle indices, and thus 
the weak dependence in Definition~\ref{def:weak_dependence} is satisfied.

\section{Numerical experiments}
\label{sec:applications}

Our numerical experiments are conducted in a Matlab environment (version R2023b), 
using 64-bit arithmetic, on a single core of a computing system equipped with an
Apple M1 Max processor and 64 GB of system memory. All numerical results reported 
in this section represent sample averages over ten independent executions with 
different random seeds. 
The matrices used throughout our experiments are constructed as follows.
\begin{itemize}
\item {\bf Discretized Laplacian:} We consider a second central difference 5-point 
and 7-point  discretization of the Laplace operator $\Delta$ on the unit square 
$\Omega=[0,1]\times[0,1]$ and unit cube $\Omega=[0,1]\times[0,1]\times[0,1]$ using 
Dirichlet boundary conditions, respectively. 
\item {\bf Model covariance:} We consider the following model covariance 
matrix, e.g., \cite{kalantzis2013accelerating,kalantzis2018scalable}: 
\begin{equation*} 
    A_{i,j} =
    \begin{cases}
      1+\sqrt{i}        & \mathrm{if}\ i==j \\
      1/|i-j|^2        & \mathrm{if}\ i\neq j
    \end{cases},
\end{equation*}
where the off-diagonal entries of $A$ decay away the main diagonal in order to 
simulate a decreasing correlation among the variables. 
\end{itemize}
For each matrix problem, we set $\gamma=\frac{1}{1.1 \rho(A)}$ and modify the resulting 
matrix as $A = \gamma A$ such that the spectral radius of the matrix $A$ is less than 
one. Moreover, in all of our experiments, the matrix $P$ is set such that 
$P_{i,j}=\frac{|A_{i,j}|}{\sum_{k=1}^d |A_{i,k}|}$.

\subsection{Error metric and number of iterations for the classical and regenerative algorithms}

Unless mentioned otherwise, we will monitor the accuracy of the classical 
(Ulam-Von Neumann) and regenerative algorithms by keeping track of the maximum (absolute) 
entry-wise error of the matrix inverse approximation. In particular, if $\widehat{C}$ denotes 
the approximation of the matrix inverse $B^{-1}$, we assess the performance of the two algorithms 
by comparing their respective modulus of $\max_{i,j}|\widehat{C}_{i,j}-B^{-1}_{i,j}|$. In order to facilitate a practical comparison between the classical and regenerative algorithms, we perform 
the same number of iterations for both approaches. In particular, recall from Section \ref{sec:bac} that the classical Ulam-Von Neumann algorithm approximates all $d$ entries of the $i$-th row 
of the matrix inverse $B^{-1}$ as
$\widehat{C}^R_{i,j}:=\frac{1}{R} \sum_{r=1}^R\left[\sum_{k=0}^{k_r}W^k_r\ONE_{[\mcmcX_r(k)=j]}\right],\ \ i\in [d],\ \ j=\{1,2,\ldots,d\}$. 
The variable $R$ denotes the total number of Markov chain replications while $k_r$ denotes 
the length of each replication. Throughout our experiments we consider eight different values for the number of replications 
$R\in 8\times [1,4,16,64,256,1024,4096,16384]$ and fix $k_r=d/4$, for a total of $dRk_r$ iterations. Similarly, 
we always set $N=dRk_r$ in Algorithm \ref{alg:ruvn}.

\subsection{Approximation of the entire matrix inverse}

\begin{figure}
    \centering
    \includegraphics[width=1.0\linewidth]{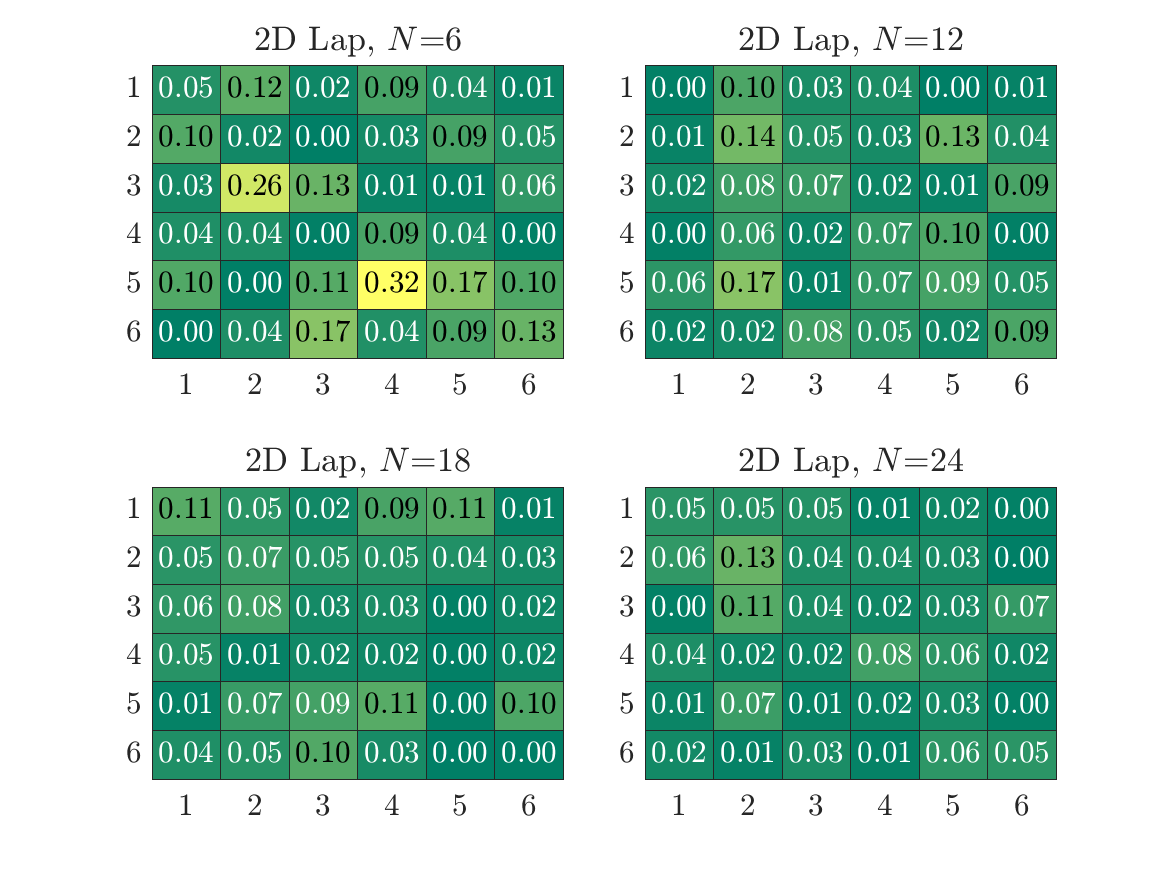}
    \vspace*{-0.25in}
    \caption{Entry-wise magnitude of the matrix $B^{-1}-\widehat{C}$ for the 5-point stencil discretization.}
    \label{fig:1}
\end{figure}

\begin{figure}
    \centering
    \includegraphics[width=1.0\linewidth]{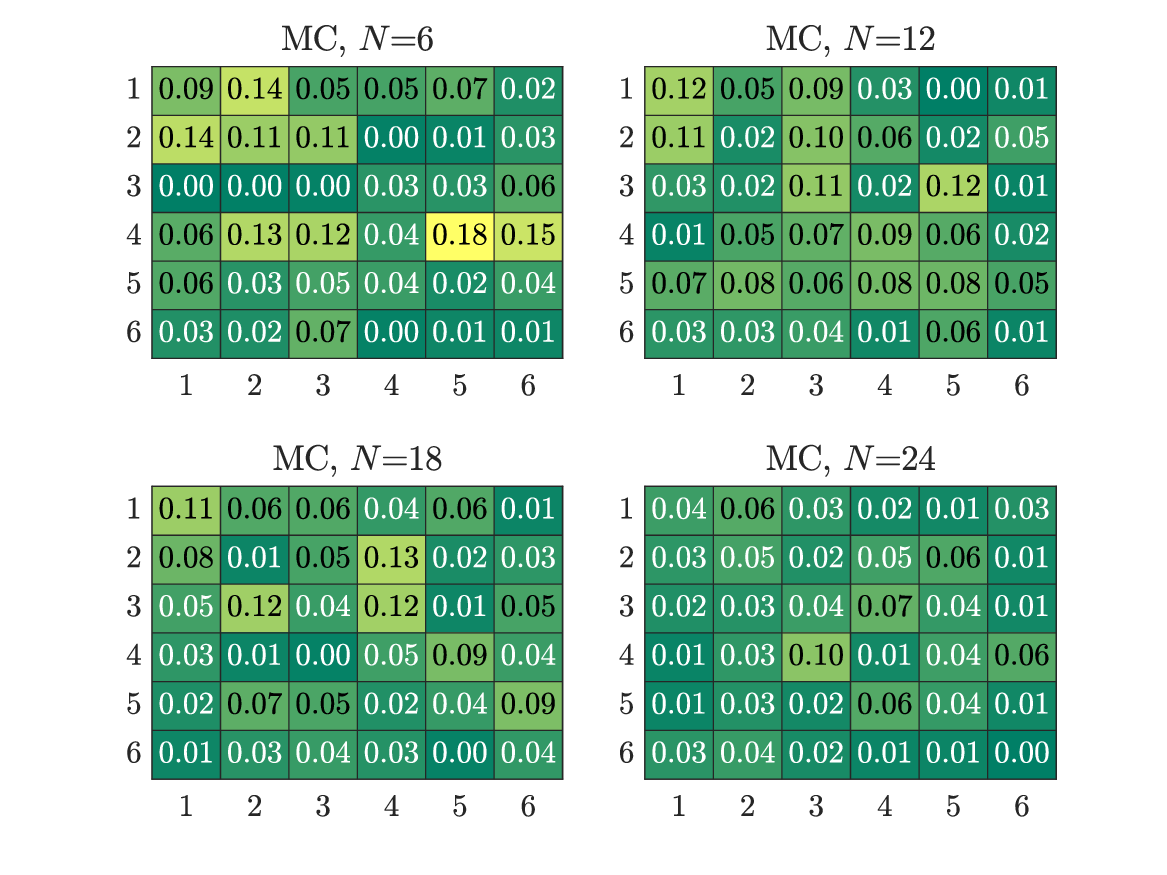}
    \vspace*{-0.25in}
    \caption{Entry-wise magnitude of the matrix $B^{-1}-\widehat{C}$ for the model covariance problem.}
    \label{fig:2}
\end{figure}
In this section we consider the approximation of the entire matrix inverse $B^{-1}$.

Figures \ref{fig:1} and \ref{fig:2} visualize the entry-wise absolute error 
$B^{-1}_{i,j}-\widehat{C}_{i,j}$ of all entries for a 5-point stencil discretization 
of the Laplacian operator and a model covariance matrix, respectively. For 
the sake of visualization, we set the dimension of each matrix equal to $d=6$ and vary
the value of $N$ manually. Larger values of $N$ result in more iterations of the 
Markov chain, and thus a statistically smaller absolute error.

We now consider the performance of the classical and regenerative algorithms on a 
set of six larger problems, divided in three subgroups as follows. The first group 
consists of two 5-pt stencil discretizations with a step size $\sqrt{d}=32$ and 
$\sqrt{d}=64$ along both spatial directions, resulting to problems of size $d=1024$ 
and $d=4096$, respectively. The second group consists of two 7-pt stencil 
discretizations where the first discretization has a fixed step size $d^{1/3}=10$ 
along all three spatial directions and the second discretization has a step size of 
twenty along the first two spatial directions and a step size of ten along the third 
direction, resulting to problems of size $d=1000$ and $d=4000$, respectively. Finally, 
the third group consists of two model covariance matrices of size $d=512$ and $d=1024$, 
respectively.

\begin{figure}
    \centering
    \includegraphics[width=0.42\linewidth]{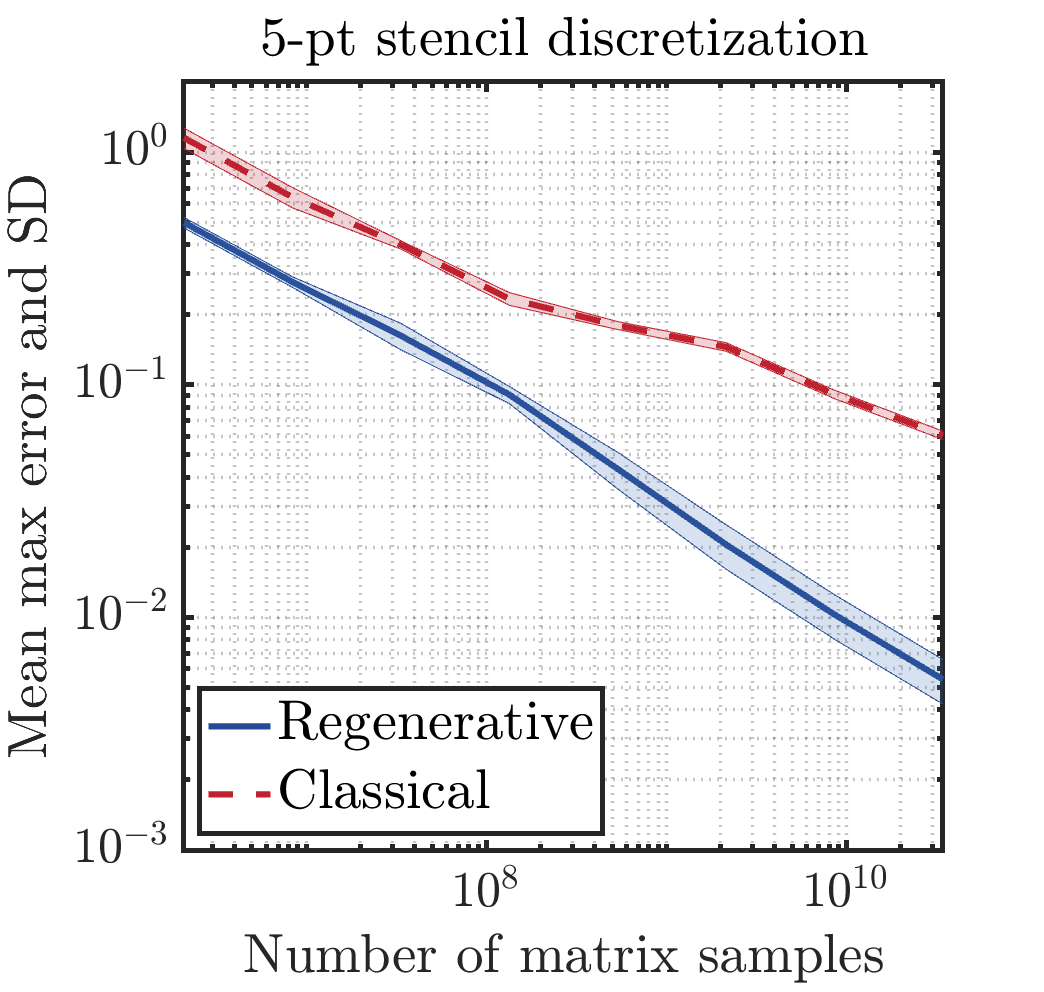}
    \includegraphics[width=0.42\linewidth]{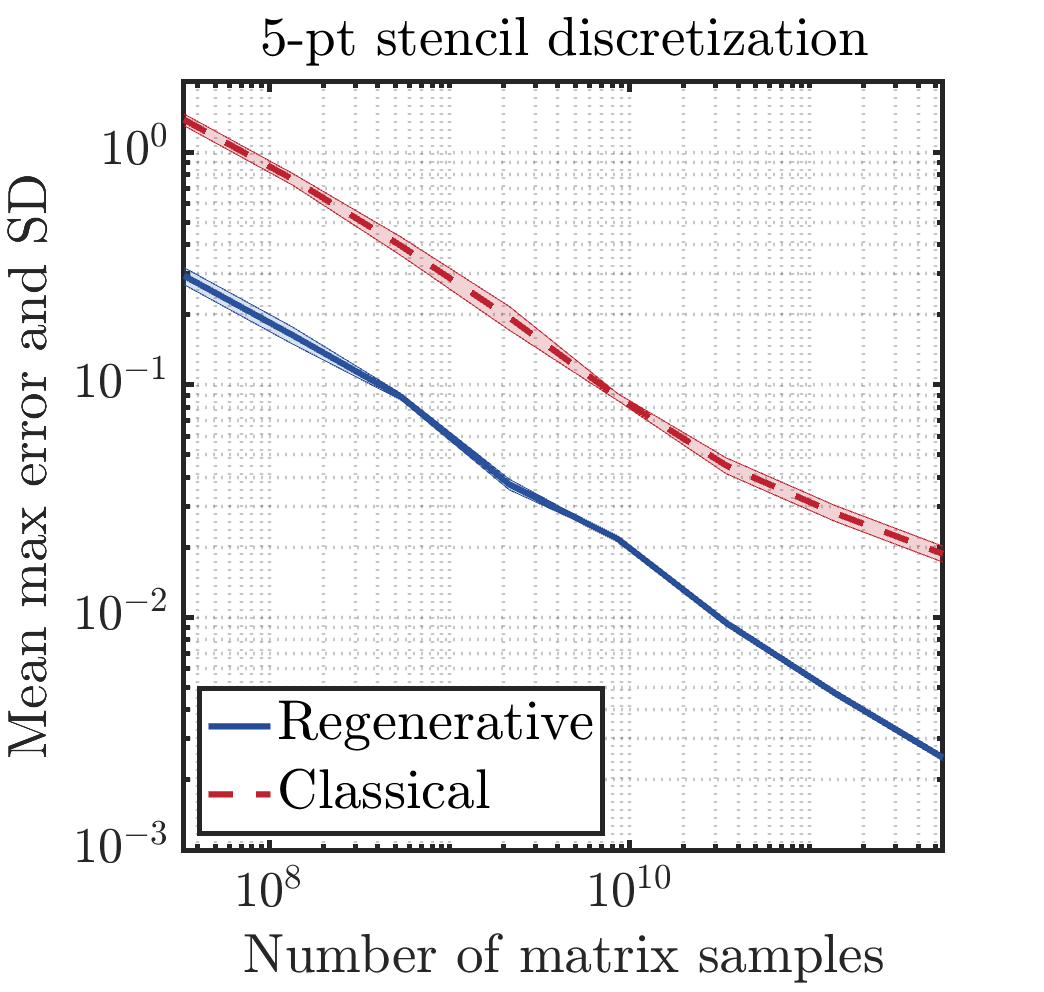}\\
    \includegraphics[width=0.42\linewidth]{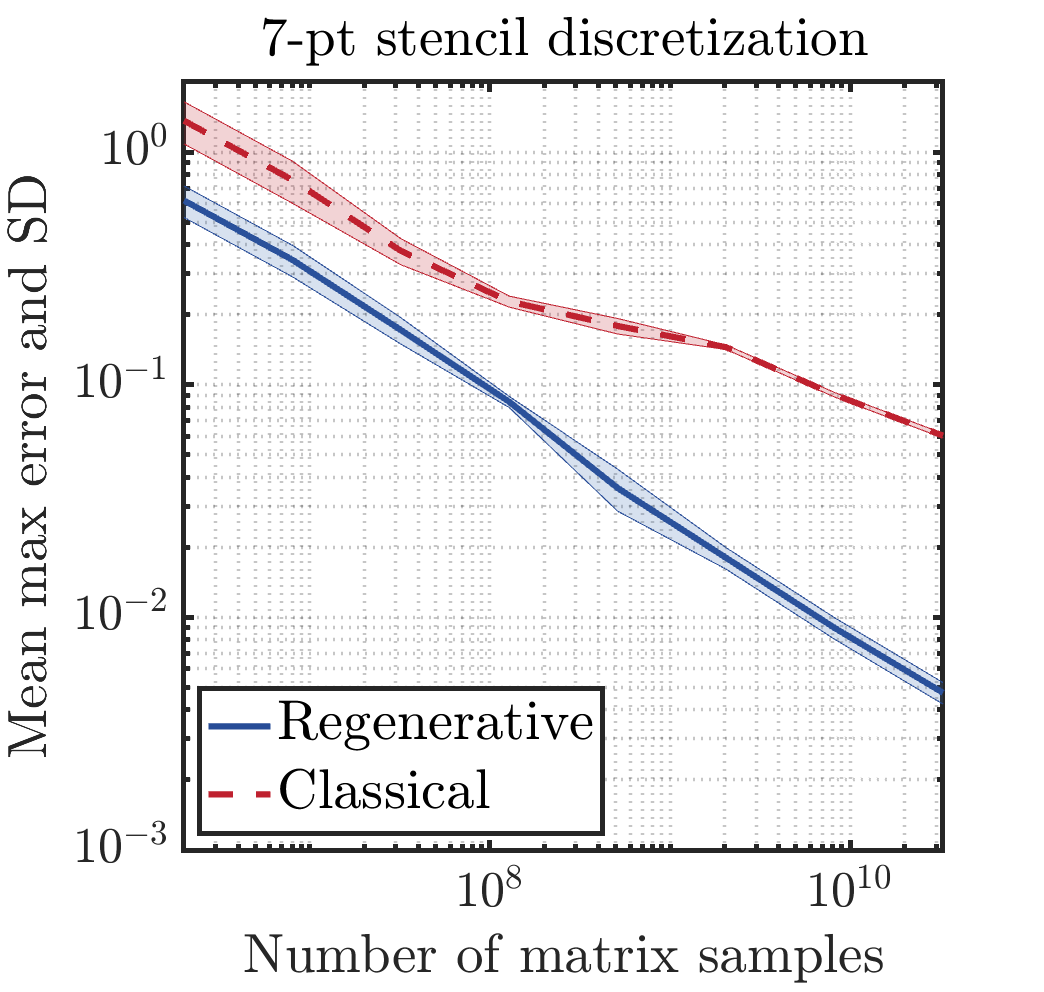}
    \includegraphics[width=0.42\linewidth]{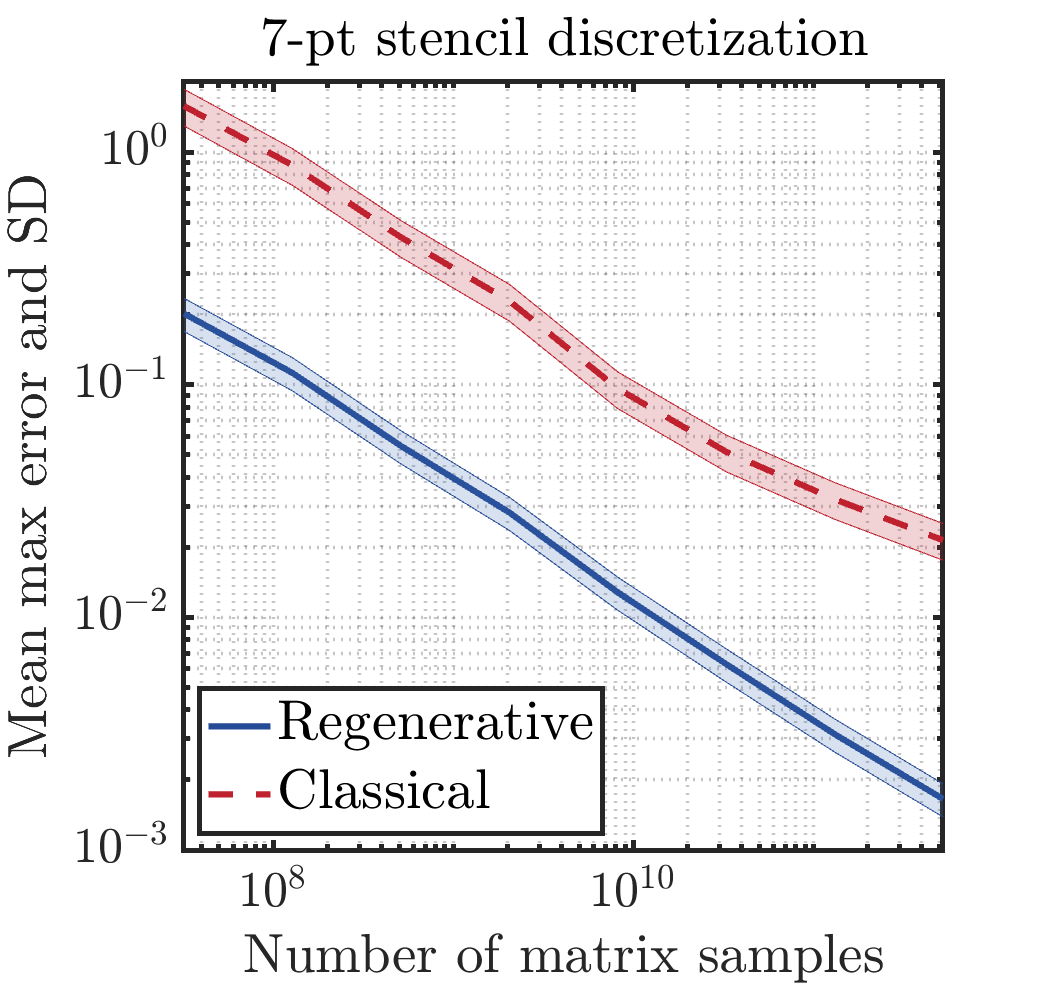}\\
    \hspace{0.099in}\includegraphics[width=0.42\linewidth]{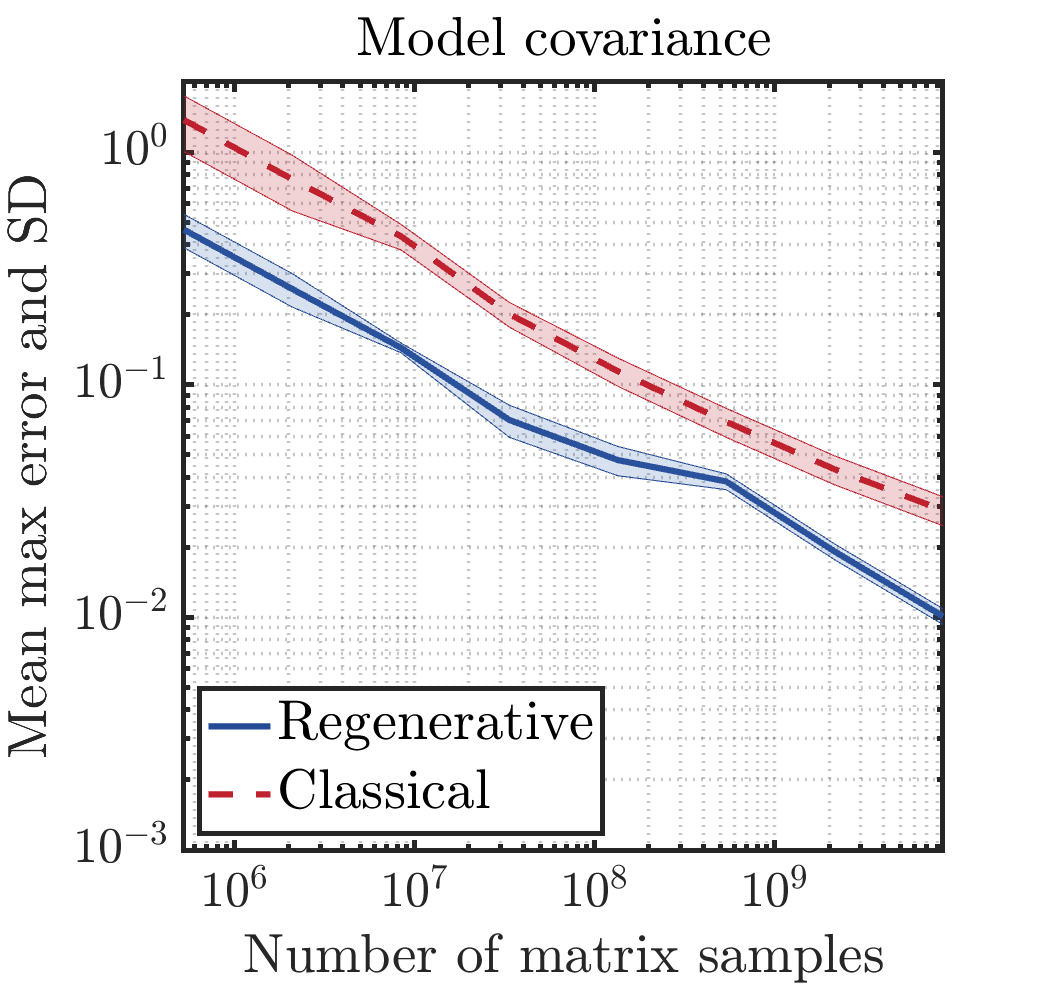}
    \includegraphics[width=0.42\linewidth]{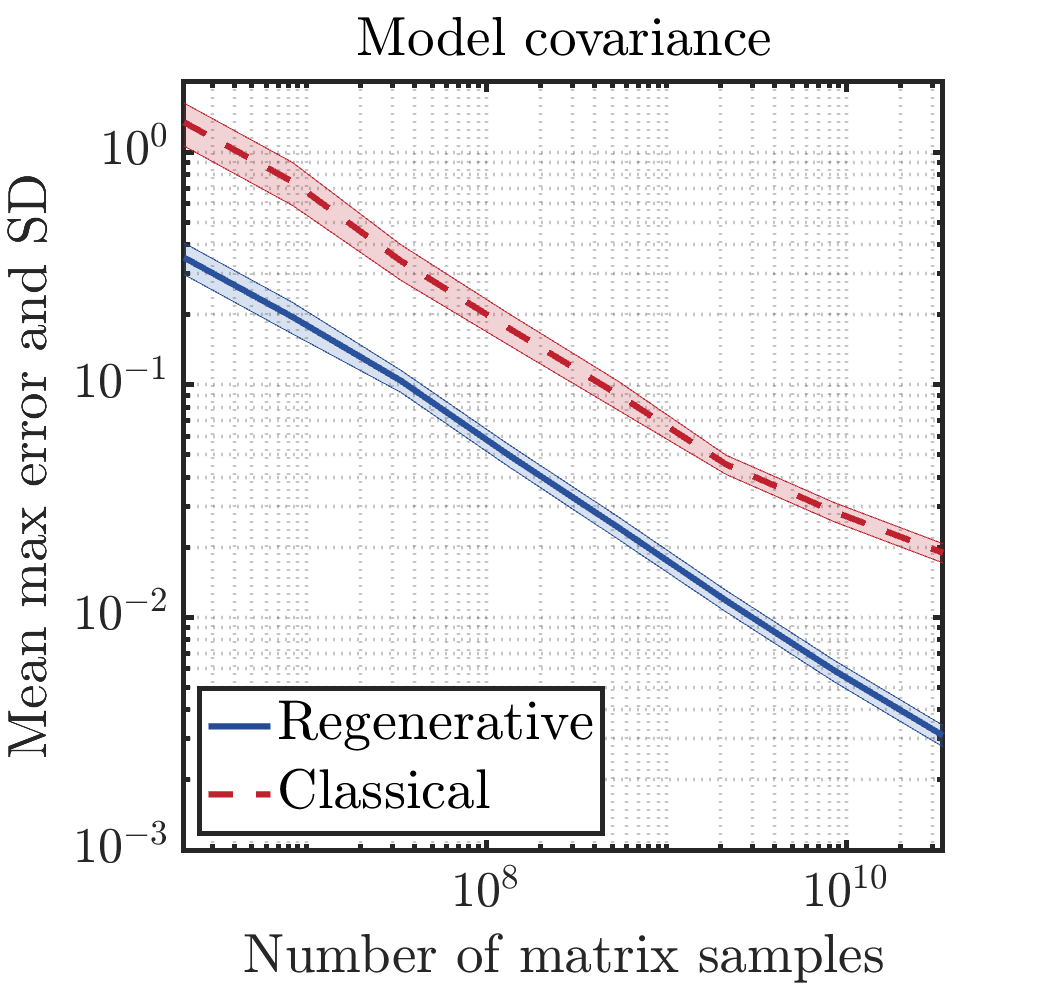}
    \caption{Maximum entry-wise error as a function of the number of entries $A_{i,j}$ sampled. 
    Top row: 5-pt stencil discretization (left: $d=1024$, right: $d=4096$). Center row: 7-pt stencil discretization (left: $d=1000$, right: $d=2000$). Bottom row: model covariances (left: $d=512$, right: $d=1024$).}\label{fig:13}
\end{figure}
Figure \ref{fig:13} plots the maximum entry-wise error achieved by the approximate matrix inverse returned by classical (dashed curve) and regenerative (solid curve) algorithms 
as the number of iterations $dRk_r$ (i.e., number of entries $A_{i,j}$ sampled) varies according 
to the eight different values of $R\in 8\times [1,4,16,64,256,1024,4096,16384]$. In addition to the mean, we also plot the standard deviation over ten independent executions with different random seeds. In summary, the regenerative algorithm returns a more accurate approximation with a corresponding error that constitutes an up to an order of magnitude improvement.

\subsection{An application in graph analytics}

We consider the application of Algorithm \ref{alg:ruvn} to determine Katz graph 
centrality, a centrality measure which extends the concept of eigenvector centrality 
by considering the influence of nodes that are connected through a path of intermediate 
nodes, i.e., beyond the immediate list of neighboring nodes \cite{benzi2014matrix}. 
Given a $d\times d$ adjacency matrix $A$, the Katz centrality of node $i$ is equal to 
$x_i = \sum_{k=1}^{\infty} \alpha^k \sum_{j=1}^{d} A_{i,j} x_j$ where the damping scalar 
$\alpha \in (0,1/\rho(A))$ controls the influence of the implicit walks. 
Gathering all centrality scores on a vector $x\in \mathbb{R}^d$, Katz centrality is 
equivalent to solving the sparse linear system $(I-\alpha A)x={\bf 1}_d$. Once the 
solution $x$ is obtained, we can rank the nodes according to their importance as indicated 
by the modulus $|x_i|,\ i=1,\ldots,d$, e.g., see \cite{kalantzis2025single} for additional 
details on walk-based centralities, and \cite{guidotti2024fast} for an existing 
application of an MCMC-type approach for Katz centralities. 
As our example, we pick the \texttt{IBM32} graph from the SuiteSparse matrix collection \cite{davis2011university}, a network from the original Harwell-Boeing sparse matrix 
test collection which represents interactions between $d=32$ leaflets 
from a 1971 IBM advertisement conference. We set $\alpha=0.85/\|A\|_2$ and call Algorithm \ref{alg:ruvn} with a varying value of $N$ to compute an approximation of 
$B^{-1}=(I-\alpha A)^{-1}$. 

The left subplot of Figure \ref{fig:5} plots the maximum entry-wise approximation error 
between $B^{-1}$ and $\widehat{C}$ as well as the maximum entry-wise error between the 
ideal Katz centrality $B^{-1}{\bf 1}_d$ and its approximation $\widehat{C}{\bf 1}_d$. 
As anticipated, performing more iterations improves the quality of the approximation. 
Moreover, approximating $B^{-1}{\bf 1}_d$ can be more challenging due to 
the accumulation of all $d$ approximation errors per row of $\widehat{C}$. The right 
subplot of Figure \ref{fig:5} plots the percentage of correctly ranked nodes (for all ten 
trials) if we rank 
the nodes of the \texttt{IBM32} graph according to the modulus of $\widehat{C}{\bf 1}_d$ 
instead of $B^{-1}{\bf 1}_d$. Performing more iterations in Algorithm \ref{alg:ruvn} leads 
to a more accurate ranking, and the rankings obtained by independent trials vary less 
as the iterations number performed by Algorithm \ref{alg:ruvn} increases.

\begin{figure}
    \centering
    \includegraphics[width=0.481\linewidth]{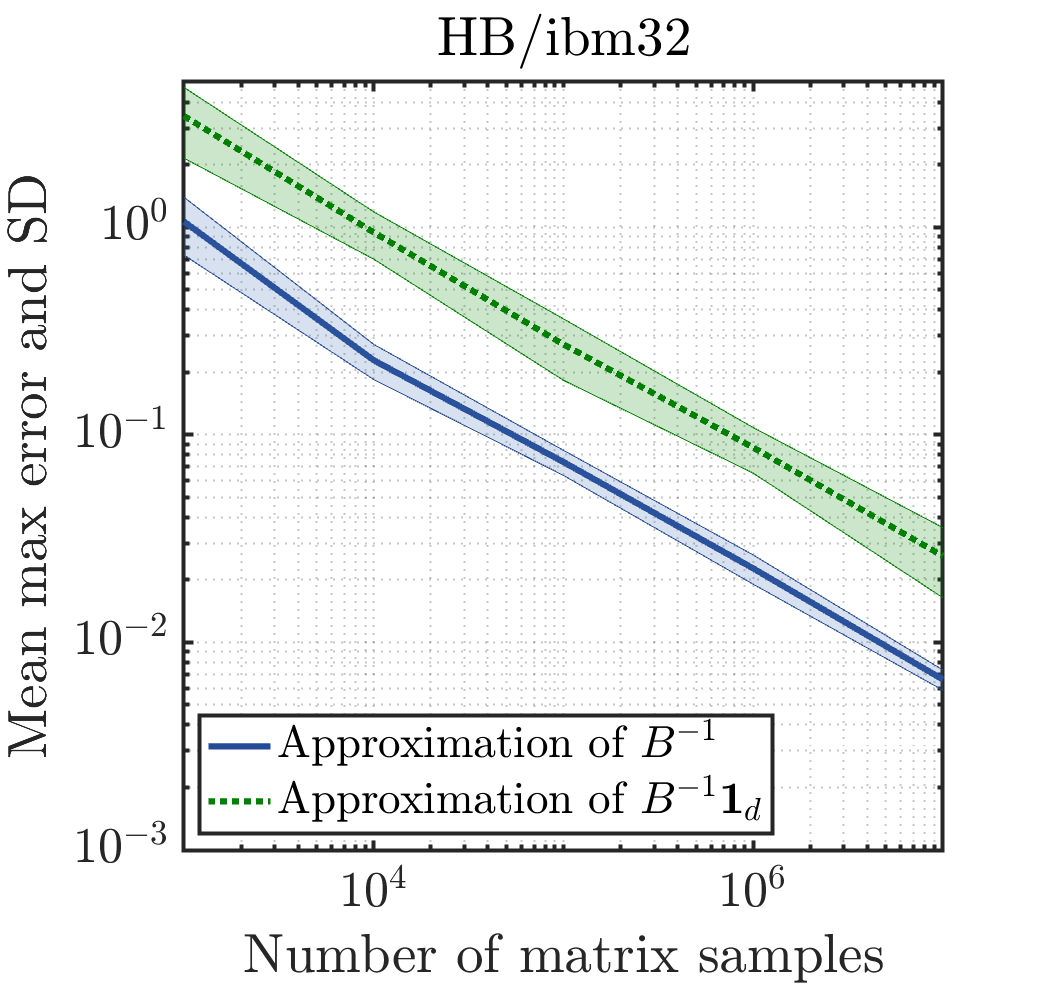}
    \includegraphics[width=0.43\linewidth]{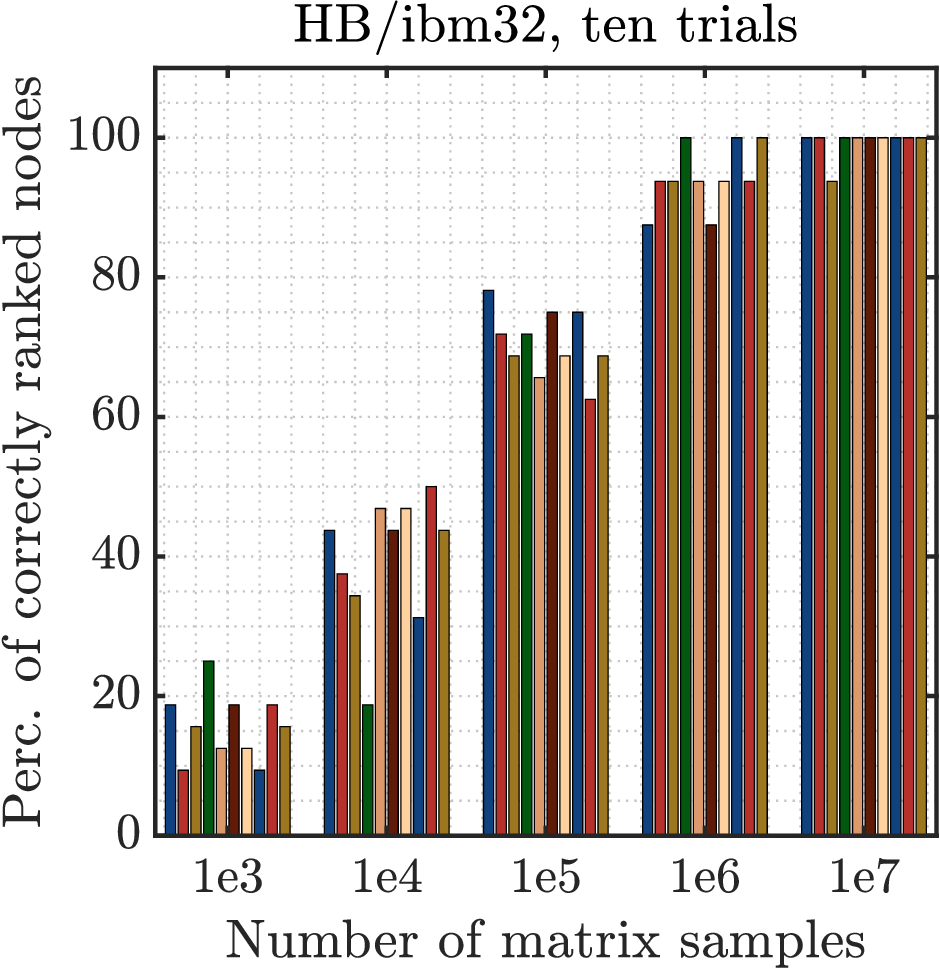}
    \vspace*{-0.1in}
    \caption{Left: maximum entry-wise error. Right: percentage of correctly ranked 
    nodes for all ten trials.}
    \label{fig:5}
\end{figure}

\subsection{Approximating a column of the matrix inverse $B^{-1}$}

\begin{figure}
    \centering
    \includegraphics[width=0.32\linewidth]{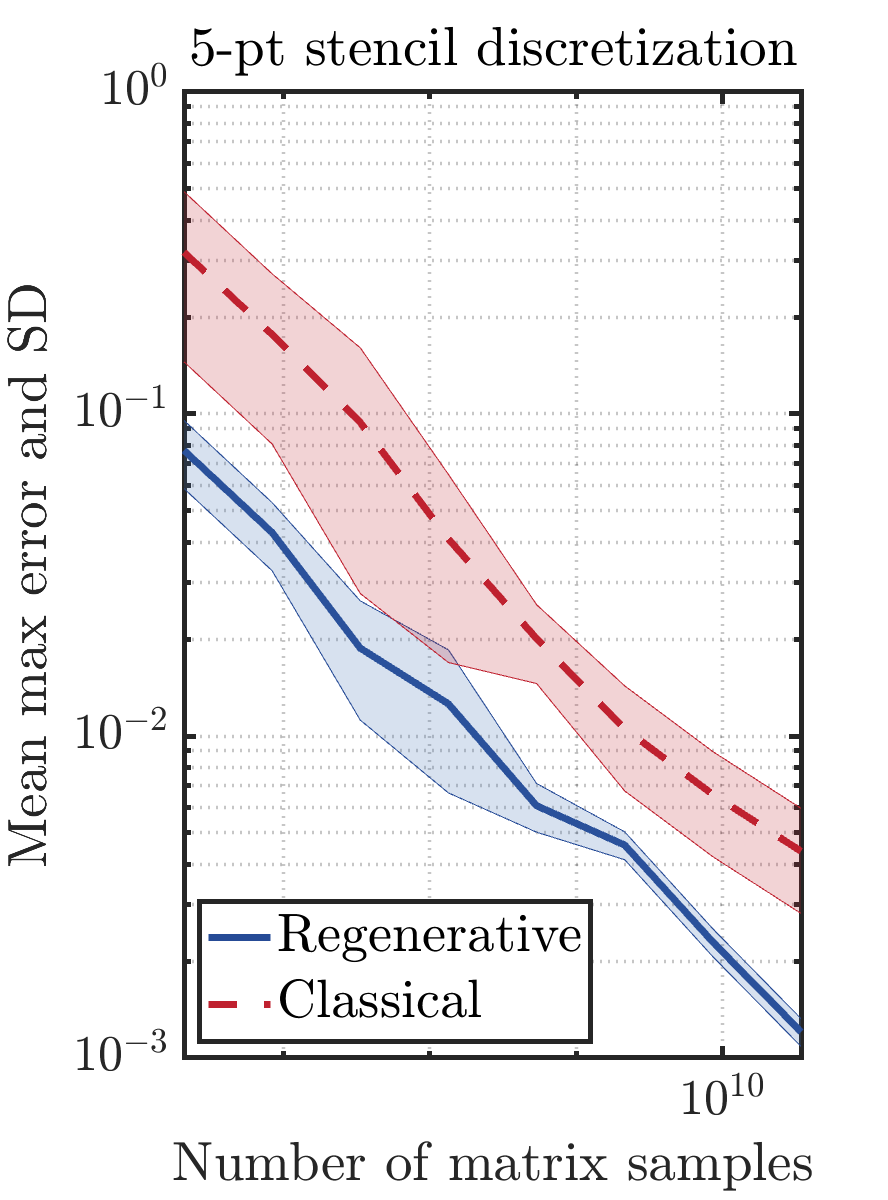}
    \includegraphics[width=0.32\linewidth]{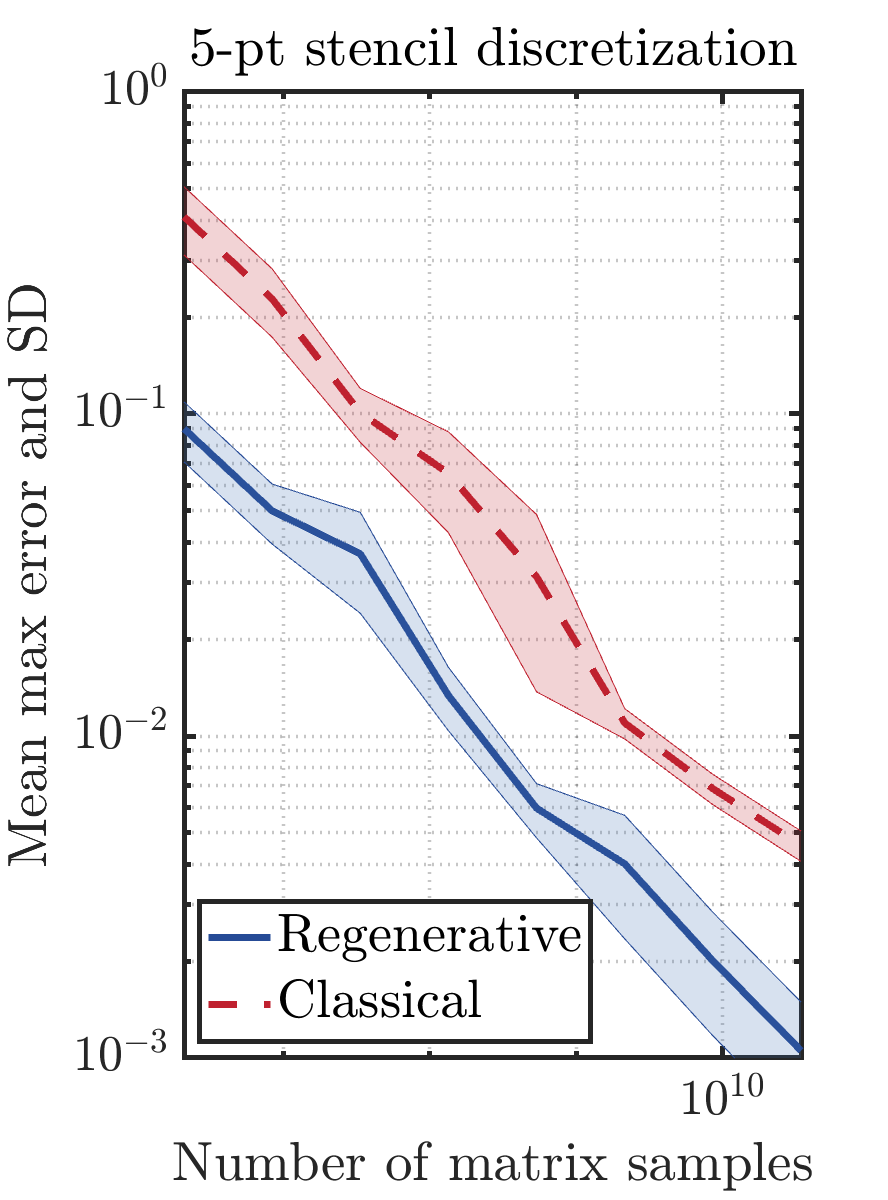}
    \includegraphics[width=0.32\linewidth]{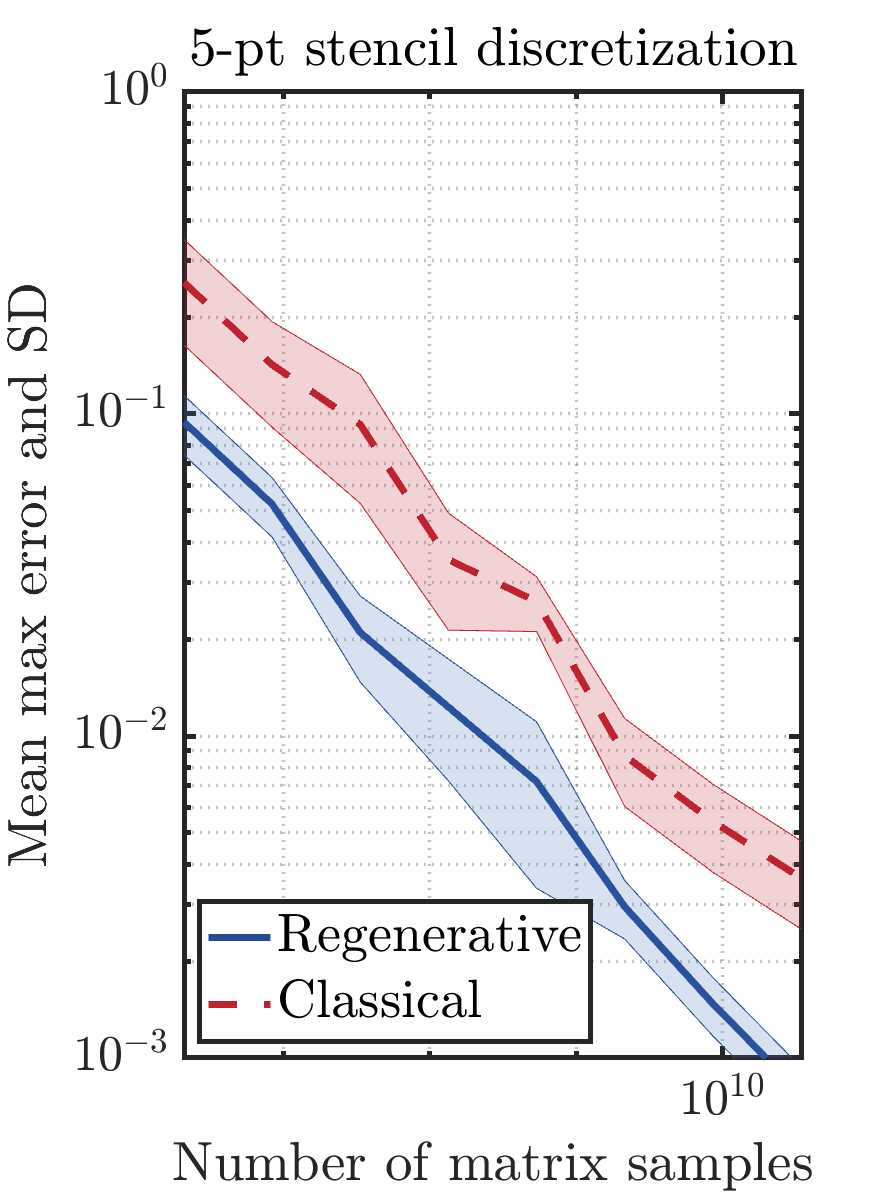}
    \caption{Maximum entry-wise error in the approximation of a column $n \in [d]$ of $B^{-1}$ achieved by classical Ulam-von Neumann and Algorithm \ref{alg:ruvn2} 
    for a 5-pt stencil discretization of size $1024$. Left: $n=1$. Center: $n=d/2$. Right: $n=d$.}
    \label{fig:col1}
\end{figure}

\begin{figure}
    \centering
    \includegraphics[width=0.32\linewidth]{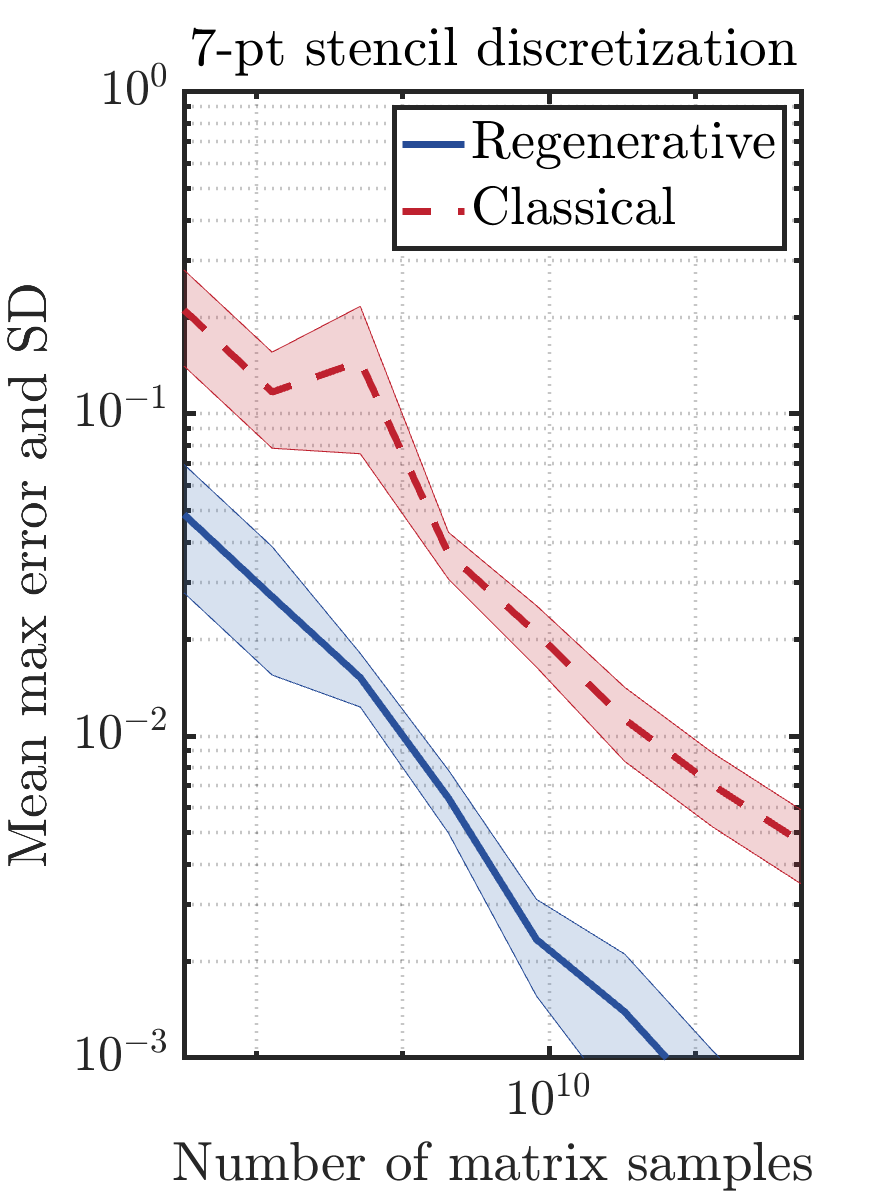}
    \includegraphics[width=0.32\linewidth]{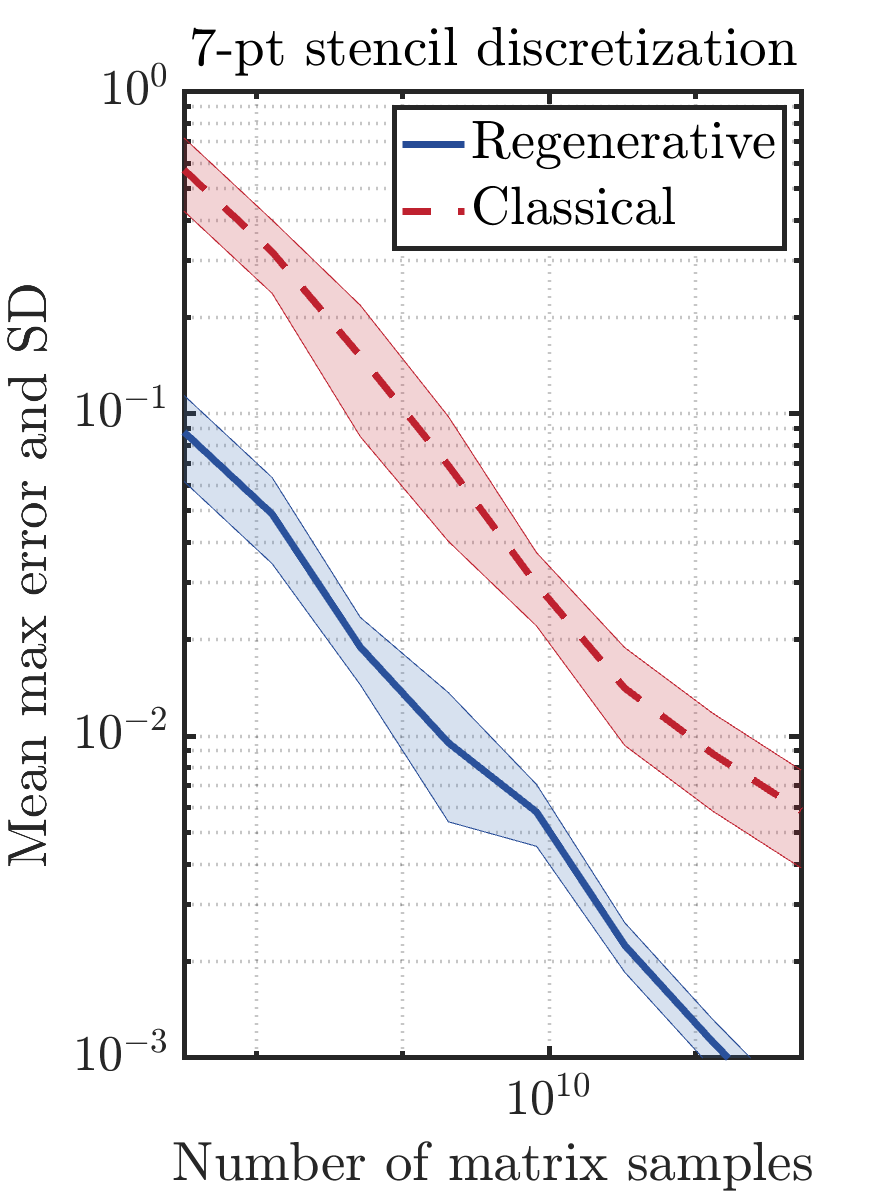}
    \includegraphics[width=0.32\linewidth]{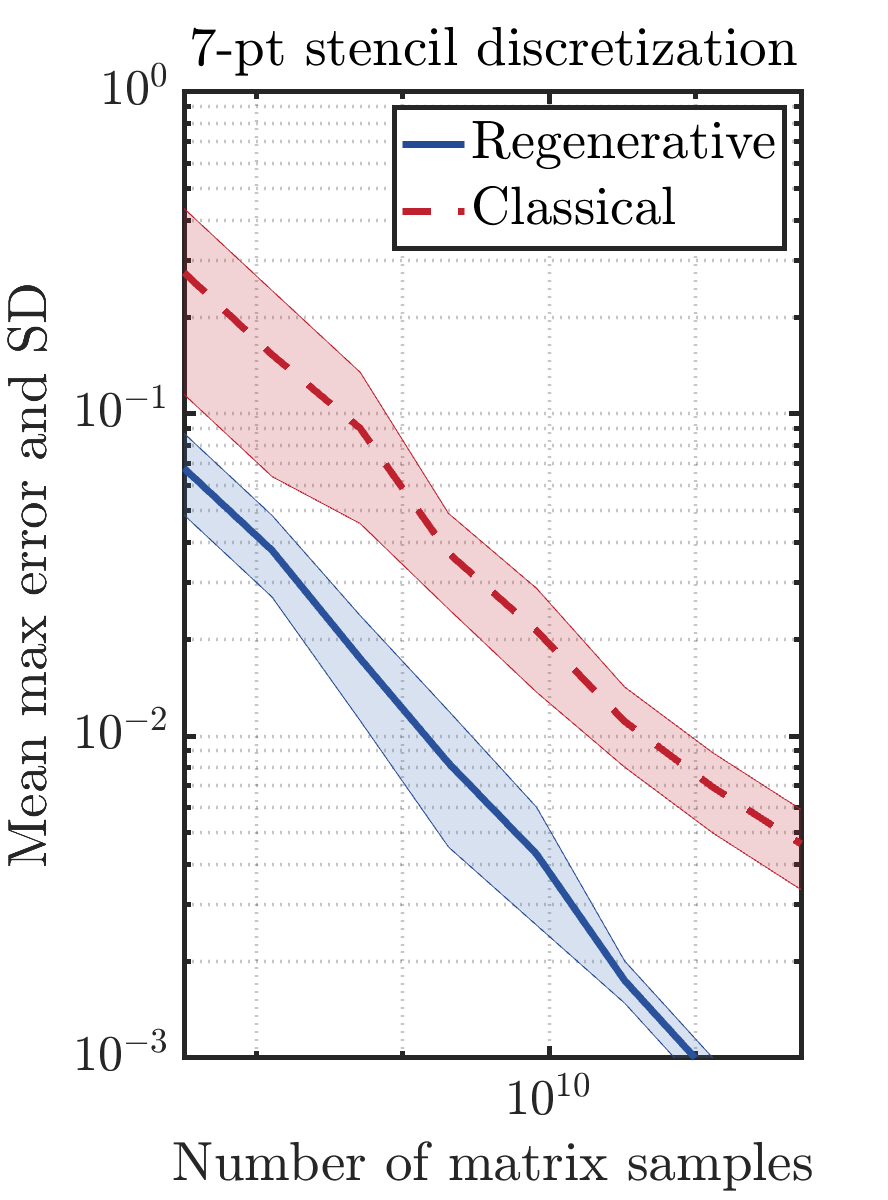}
    \caption{Maximum entry-wise error in the approximation of a column $n \in [d]$ of $B^{-1}$ achieved by classical Ulam-von Neumann and Algorithm \ref{alg:ruvn2} 
    for a 7-pt stencil discretization of size $4000$. Left: $n=1$. Center: $n=d/2$. Right: $n=d$.}
    \label{fig:col2}
\end{figure}

In this section we consider the approximation of an individual column of the matrix inverse $B^{-1}$ via both the classical Ulan-Von Neumann algorithm and the regenerative approach as outlined in Algorithm \ref{alg:ruvn2}. Such computations can arise, for example in personalized Katz centrality where we need to solve the system of linear equations $(I-\alpha A)x=e_n$ where $e_n\in \mathbb{R}^d$, $n \in [d]$, denotes the $n$-th column of the identity matrix of dimension $d$. The solution of this system is simply the $n$-th column of the matrix inverse $B^{-1}$ \cite{10.1007/978-3-319-78024-5_26}. 
Computing individual columns of $B^{-1}$ is a particularly suitable setting for Algorithm \ref{alg:ruvn} since each time the regenerative algorithm returns to state $n\in [d]$, the 
entire $n$-th column of $B^{-1}$ is updated as opposed to updating only a single entry via the 
classical Ulam-von Neumann algorithm. When the matrix $A$ is not explicitly known and the computation of a 
random entry $A_{i,j}$ is expensive, it is highly desirable to provide an approximation of $B^{-1}$ that leverages as few sampled entries $A_{i,j}$ as possible. For example, consider 
$A=FF^\top$ for some matrix $F\in \mathbb{R}^{d\times d}$. Then, computing an entry $A_{i,j}$ 
can incur a computational cost of $2d-1$ operations which is an order $O(d)$ greater than the computational cost of each iteration of the classical Ulam-von Neumann matrix inversion algorithm, e.g., see also the related discussion in Section 5.4 in \cite{doi:10.1137/130904867}. 

Figure \ref{fig:col1} plots the maximum entry-wise error of the approximation of a column $n \in [d]$ of $B^{-1}$ achieved by classical Ulam-von Neumann and Algorithm \ref{alg:ruvn2} for a 5-pt stencil discretization of dimension $d=1024$. We consider three different scenarios where $n \in \{1,d/2,d\}$, i.e., we approximate the first, middle, and last column of $B^{-1}$. The maximum entry-wise errors are plotted against the number of samples $A_{i,j}$ which is also equivalent to the total number of iterations by each approach. In all cases, Algorithm \ref{alg:ruvn} achieves a lower error due to exploiting the regenerative structure of the Markov chain. Figure \ref{fig:col2} plots the same quantities as in Figure \ref{fig:col1} for a 7-pt stencil discretization of dimension $d=4000$.

\section{Conclusion}
\label{sec:conclusions}

This paper presented a Markov chain Monte Carlo algorithm to approximate the inverse of a matrix with convergent Neumann series. The proposed algorithm recasts the classical problem into one of a stochastic fixed-point equation induced by the regenerative structure of the Markov chain. The  estimator associated with the regenerative algorithm does not require any truncation of the Neumann series.  
Moreover, only one parameter -the total number of simulated Markov transitions- is necessary, as opposed to the two parameters required in the Ulam-von Neumann algorithm. A probabilistic analysis of the proposed algorithm is conducted and a central limit theorem is obtained. Numerical experiments demonstrate that the proposed algorithm can achieve a more accurate approximation of the entries of the matrix inverse compared to classical Ulam-von Neumann algorithm. 

As part of our future work we plan to implement Algorithm \ref{alg:ruvn} on shared memory high-performance computing architectures such as Graphics Processing Units (GPUs), and compare its performance against classical Markov chain Monte Carlo approaches. For example, in contrast to the algorithm of Ulam-von Neumann which relies on scalar multiplications, the main computational cost of Algorithm \ref{alg:ruvn} stems from matrix-scalar multiplication of the form $\Lam=\Lam\times \frac{A_{i,j}}{P_{i,j}}$, an operation that can be parallelized via GPUs. Other interesting directions include the application of the regenerative approach in obtaining approximate inverse blocks in block-Jacobi preconditioners, and explore combinations of the work presented in this paper with the MCMC algorithm in \cite{guidotti2024fast}. Finally, we are planning to explore how different properties of the matrix $A$ affect the practical performance of the regenerative algorithm, and mechanisms to set the transition matrix $P$ so that the variance of the estimator is reduced.

\section*{Acknowledgments}

The authors would like to express their gratitude to the editor and the two anonymous reviewers for their meticulous reading and detailed comments which led to several improvements in our initial submission. We are indebted to one of the reviewers for their continuous effort and numerous suggestions to improve factual correctness and clarity of our submission. The authors would also like to thank Prof. Michele Benzi and Prof. Yuanzhe Xi for their discussions and comments regarding the work presented in this manuscript.


\section*{Appendix A: Proof of Theorem~\ref{thm:main_clt}} \label{appendixproof}

\begin{proof}[Proof of Theorem~\ref{thm:main_clt}]
Let $p_n$ and $q_n$ be two sequences of positive integers satisfying $q_n\ra \infty$, $q_n=o(p_n)$, $p_n =O(n^{1-\beta})$ for a fixed $\beta\in (0,1)$ and $q_n^{(2m-3)/(m-2)}=O(n)$ as $n\ra \infty$. This means that both $p_n$ and $q_n$ go to infinity but slower than $n$, moreover $q_n$ grows slower than $p_n$. Define $k_n = \left[\frac{n}{p_n+q_n}\right]$, with $[\cdot]$ denoting the Gauss bracket, and
\begin{align*}
C^n_j=& ((p_n+q_n)(j-1)+p+1,(p_n+q_n)j] \cap \Natural,\\ 
B^n_j = &((p_n+q_n)(j-1), (p_n+q_n)(j-1)+p_n]\cap \Natural,\\ 
V^n_j=& \sum_{i\in C^n_j} X_i,\\
U^n_j = &\sum_{i\in B^n_j} X_i,\\ 
Z_n= &\frac{1}{\sqrt{n}} (U^n_1+\ldots+U^n_{k_n}).
\end{align*}
Following the above definitions, the distance between any two blocks $B_j^n$ and $B_k^n$ 
$j\neq k$, is at least $q_n$, while $C_j^n$ contains the integers between $B^n_{j-1}$ and 
$B^n_{j}$. Let now
\begin{align*}
{\widetilde Z}_n=&\frac{1}{\sqrt{n}} \left({\widetilde U}^n_1+\ldots+{\widetilde U}^n_{k_n}\right),\\  Y_n =& \frac{1}{\sqrt{n}} \left({\widetilde V}^n_1+\ldots+{\widetilde V}^n_{k_n}\right),
\end{align*}
with ${\widetilde U}^n_i$ independent random variables following the same distribution as $U^n_i$, and ${\widetilde V}^n_i$ independent Gaussian variables with the same first two moments as $U^n_i$. For a fixed $t\in \Real^d$ in a compact set $\Omega$ containing the origin, define $f:\Real^d\ra \Complex, x\mapsto \exp(\sqrt{-1} (t \cdot x))$. Then, the difference between the characteristic functions of ${\widehat X}_n$ and the limit $X_\infty$ is a random variable distributed as $N(0, M^2)$ and 
can be expressed as
\begin{align*}
&\ex \left[f({\widehat X}_n)-f(X_\infty)\right] \\= &\ex\left[f({\widehat X}_n)-f(Z_n)+f(Z_n)-f({\widetilde Z}_n)+f({\widetilde Z}_n)-f(Y_n)+f(Y_n)-f(X_\infty )\right] \\ =& \underbrace{\ex[f({\widehat X}_n)-f(Z_n)]}_{I}+\underbrace{\ex[f(Z_n)-f({\widetilde Z}_n)]}_{II}+\underbrace{\ex[f({\widetilde Z}_n)-f(Y_n)]}_{III}+\underbrace{\ex[f(Y_n)-f(X_\infty )]}_{IV}.
\end{align*}
Furthermore, following the Taylor series expansion, we know that for any $x, y\in \Real^d$, $f(x)-f(y) = \nabla f(\xi)\cdot x-y$ for some $\xi\in\Real^d$ with $\nabla f(x) =  \exp(\sqrt{-1}(t \cdot x)) \sqrt{-1}t$. Hence, 
\begin{align}
\label{eqn:taylor}
\|f(x)-f(y)\| \le \|\nabla f(\xi)\|_2 \|y-x\|_2= \|t\|_2 \|y-x\|_2.
\end{align}

The process of estimating the above four quantities and showing that they converge to zero as $n\ra \infty$ (thus proving the central limit theorem) is known as the Bernstein block technique of the \emph{Lindeberg method}, e.g., see~\cite{petrov1995limit} and~\cite{rio1997lindeberg}. Terms $II$ and $III$ are known to be the \emph{main} terms while terms $I$ and $IV$ are often referred to as the \emph{auxiliary} terms. In the following, we show that each one of these four terms converges to zero as $n\ra \infty$. 

\subsubsection*{Calculation of the term $I$}
Following the definition of the function $f$ and inequality~\eqref{eqn:taylor}, we can write
\begin{align*}
|I|= & |\ex[f({\widehat X}_n)-f(Z_n)]|& & 
\\
\le& \|t||_2 \ex\left[\bigg\| \frac{1}{\sqrt{n}} (V^n_1+\ldots+V^n_{k_n})\bigg\|_2 \right]
\\ 
\le&  \frac{2\|t\|_2 }{\sqrt{n}} \left[\sum_{j=1}^{k_n}\sum_{i\in C_j^n} \bigg\|\sum_{\ell \ge i}\ex[X_i X_\ell]\bigg\|_2\right]^\frac12
\\
\le&
\frac{2\|t\|_2 }{\sqrt{n}} (k_n q_n +p_n)^\frac12 \|M^2\|_2.
\end{align*}
Moreover, from the assumptions made on $p_n$, $q_n$ and $k_n$, we know that 
\begin{equation*}
    \frac{k_n q_n +p_n}{n}\le \frac{k_n q_n +p_n}{k_n (q_n +p_n)}= \frac{k_n q_n }{k_n (q_n +p_n)}+\frac{p_n}{k_n (q_n +p_n)} \le \frac{ q_n }{q_n +p_n}+\frac{1}{k_n}.
\end{equation*}
Therefore, $|I|$ tends to zero as $n\ra \infty$.

\subsubsection*{Calculation of the term $II$}

The term $II$ is controlled by the weak dependency condition while the correlations are 
controlled by the Lipschitz constant and $\epsilon_r$. In particular, 
\begin{equation*}
    II= \ex\left[f(Z_n)-f\left({\widetilde Z}_n\right)\right] \le \sum_{j=1}^{k_n} \left|\ex\left[\Delta^n_j\right]\right|,
\end{equation*}
with
\begin{align*}
\Delta^n_j=&f(W_j^n+x^n_j)- f(W_j^n+{\widetilde x}^n_j),  \\x^n_j= &\frac{1}{\sqrt{n}} U^n_j,  \qquad \ \ \mathrm{and}\ \ {\widetilde x}^n_j= \frac{1}{\sqrt{n}} {\widetilde U}^n_j,\\ W^n_j=&\sum_{i<j}x^n_i +\sum_{i>j}{\widetilde x}^n_i\,.
\end{align*}
The exponential function form of $f$ then indicates that
\begin{align*}
\left|\ex\left[\Delta^n_j\right]\right|\le \Bigg|\cov \left(f\left(\sum_{i<j}x^n_i\right), f(x^n_j)\right)\Bigg|.
\end{align*}
Then, by the assumption that $\{X_n\}$ is a $({\epsilon_r}, \psi)$-weakly dependent sequence
(${r\in\bZ}$), we have $|\ex\left[\Delta^n_j\right]| \le C k_np_n \epsilon_{q_n}$. Thus, the condition of $\epsilon_r$, $p_n$ and $q_n$ leads to the desired result. 

\subsubsection*{Calculation of the term $III$}

The term $III$ can be written as 
\begin{align*}
III=&\ex[f({\widetilde Z}_n)-f(Y_n)]\le \sum_{j=1}^{k_n} \left|\ex\left[\widetilde \Delta^n_j\right]\right|
\qquad\text{with}
\\
{\widetilde \Delta}^n_j=&f({\widetilde W}_j^n+{\widetilde x}^n_j)- f({\widetilde W}_j^n+{\widehat x}^n_j), \qquad {\widehat x}^n_j= \frac{1}{\sqrt{n}} {\widetilde V}^n_j, \qquad {\widetilde W}^n_j=\sum_{i<j}{\widetilde x}^n_i +\sum_{i>j}{\widehat x}^n_i.
\end{align*}
Then, $\left|\ex\left[\widetilde \Delta^n_j\right]\right|= |\ex f({\widetilde x}^n_i)-\ex f({\widehat x}^n_i)|$. 
Due to Taylor expansion, there exists $C>0$, such that for all $x\in \Real^d$ and $t\in \Omega\subset \Real^d$, and any $\al$ such that $\beta(1+\delta/2)\le 1+\al/2$,

\begin{align*}
\bigg\|f(x) - \left[1+ i t\cdot x - \frac12 x^T (tt^T) x) \right]\bigg\|\le C \|x\|_{2+\alpha}^{2+\al}.
\end{align*}
Therefore, the moment condition has the term $III$ bounded by $Cn^{-1-\frac{\al}{2}} k_n^{1+\frac{\delta}{2}}$ by applying Lemma~\ref{lem:delta_moment} (shown below) with $m=k_n$ to each coordinate, 
and it goes to zero as $n\ra \infty$ as a result of the way $\al$ is selected.

\subsubsection*{Calculation of the term $IV$}
For the term $IV$, recall that the independent Gaussian variables ${\widetilde V}^n_i$ have the same first two moments as $U^n_i$, i.e. $\ex[{\widetilde V}_i]=0$, and $\ex[{\widetilde V}^n_i({\widetilde V}^n_i)^T]=\ex[U^n_i(U^n_i)^T]$, and they are independent of index $i$. Then, $\ex[f(Y_n)-f(X_\infty )]\ra 0$ as $n\ra \infty$, which can be verified through direct calculations on the second moments of $Y_n$ since both $Y_n$ and $X_\infty$ are Gaussian.
\end{proof}

\begin{lemma}
\label{lem:delta_moment}
For i.i.d. random variables $X_i\in \Real$ with finite $2+\delta$ moment for some $\delta>0$, there exists $C>0$ independent of $X$, such that, for any integer $m>1$ and any $0<\al <\delta$, we have, 
\begin{align}
\label{eqn:delta_moment}
\ex\left[\bigg|\sum_{i=1}^m X_i\bigg|^{2+\al}\right]\le Cm^{\frac{2+\delta}{2}}
 (\ex[|X_1|^{2+\delta}]).
\end{align}
\end{lemma}
\begin{proof}
First, consider the case where $X_i$ are symmetric. Then, for any $\al\le \delta$,
\begin{align*}
\ex\left[\bigg|\sum_{i=1}^m  X_i\bigg|^{2+\al}\right]=&\ex\left[\left(\sum_{i=1}^m  X_i\right)^2\bigg|\sum_{k=1}^m  X_k\bigg|^{\al}\right]
\\=&\sum_{i=1}^m \ex\left[X_i^2\bigg|\sum_{j=1}^m  X_j\bigg|^{\al}\right]+ \sum_{i,j=1}^m \ex\left[X_iX_j \bigg|\sum_{i=1}^m X_i\bigg|^{\al}\right]\\=&m\ex\left[X_1^2\bigg|\sum_{j=1}^m  X_j\bigg|^{\al}\right],
\end{align*}
where the last step follows from the symmetry and i.i.d. assumptions. For $\al<\delta$, H\"older's inequality gives us,
\begin{align*}
\ex\left[X_1^2\bigg|\sum_{j=1}^m  X_j\bigg|^{\al}\right]\le &(\ex[|X_1|^{2+\delta}])^{\frac{2}{2+\delta}}
\left(\ex\left[\bigg|\sum_{j=1}^m  X_j\bigg|^{\frac{\al(2+\delta)}{\delta}}\right]\right)^{\frac{\delta}{2+\delta}}\\
\le & (\ex[|X_1|^{2+\delta}])^{\frac{2}{2+\delta}}
\left(\ex\left[\bigg|\sum_{j=1}^m  X_j\bigg|^{2+\alpha}\right]\right)^{\frac{\delta}{2+\delta}}, 
\end{align*}
because $\frac{\al(2+\delta)}{\delta}< 2+\al$. 
Therefore, we have, 
\begin{align*}
\ex\left[\bigg|\sum_{i=1}^m  X_i\bigg|^{2+\al}\right]
&\le 
m (\ex[|X_1|^{2+\delta}])^{\frac{2}{2+\delta}}
\left(\ex\left[\bigg|\sum_{j=1}^m  X_j\bigg|^{2+\alpha}\right]\right)^{\frac{\delta}{2+\delta}},
\end{align*}
So,
\begin{align*}
\left(\ex\left[\bigg|\sum_{j=1}^m  X_j\bigg|^{2+\alpha}\right]\right)^{\frac{2}{2+\delta}}
&\le 
m (\ex[|X_1|^{2+\delta}])^{\frac{2}{2+\delta}}.
\end{align*}
Thus,
\begin{align}
\label{eqn:delta_moment_sym}
\ex\left[\bigg|\sum_{j=1}^m  X_j\bigg|^{2+\alpha}\right]
&\le 
m^{\frac{2+\delta}{2}} (\ex[|X_1|^{2+\delta}])\,.
\end{align}
For a general sequence $X_i$, applying a Hoffman-J\o{}rgensen-type inequality, see, e.g. Lemma 1.2.6 in~\cite{de1999decoupling}, implies constants $C_1, C_2>0$, such that 
\[C_1\ex\left[\bigg|\sum_{i=1}^m X_i\bigg|^{2+\delta}\right]\le \ex\left[\bigg|\sum_{i=1}^m  \epsilon_i X_i\bigg|^{2+\delta}\right]\le C_2\ex\left[\bigg|\sum_{i=1}^m  X_i\bigg|^{2+\delta}\right],\] where $\epsilon_i,\  i=1,2,\ldots,n$ is a Rademacher sequence, i.e., i.i.d. random variables taking either $1$ or $-1$ with equal probabilities. Thus,   \eqref{eqn:delta_moment} is obtained through applying inequality~\eqref{eqn:delta_moment_sym} to the Rademacher average $\sum_{i=1}^m  \epsilon_i X_i$ because the random variables $\epsilon_i X_i$ are symmetric.
\end{proof}

\section*{Appendix B: Computation of a single column of the matrix inverse} \label{appendixa}

\begin{algorithm}
    \setcounter{AlgoLine}{0}
    \SetKwInOut{Input}{Input}
    \SetKwInOut{Output}{Output}
    \Input{$A\in \mathbb{R}^{d\times d},\ P\in \mathbb{R}^{d\times d}$, scalar $N\in \mathbb{N}$, and sought column index $n\in [d]$. Also define and set the column vectors 
    $\Lam\in \mathbb{R}^{d},\ \Sigma\in \mathbb{R}^{d},\ \Gamma\in \mathbb{N}^{d}$ equal to zero.     }
    {\bf Initialization:} Initialize the Markov chain at some random state $i\in [d]$. \\
    \While{$\min_q \Gamma_q< N$}{
  \  \emph{Update the $i$-th row of $\Lam$ if $\Lam_{i}=0$}:
    $\Lam_{i}=\Lam_{i}+\ONE_{[\Lam_{i}=0]}$;
    \\
    \emph{Sample next state $j$ of Markov Chain using probability vector $P_{i,\cdot}$}
    \\
    \emph{Update $\Lam$:}
    $\Lam=\Lam\times \frac{A_{i,j}}{P_{i,j}}$;
    \\
    \emph{If $n\equiv j \rightarrow$ Update $\Gamma$:} 
    $\Gamma_{k}=\Gamma_{k}+1$ if $\Lam_{k}\neq 0,\ k=1,2,\ldots,d$;
    \\
    \emph{If $n \equiv j \rightarrow$ Update $\Sigma$:} 
    $\Sigma=\Sigma+ \Lam$;
    \\
    \emph{If $n \equiv j \rightarrow$ Reset $\Lam$: }
    $\Lam=0$; 
    \\
    \emph{Over-write $i\leftarrow j$.}
    }
    \Output{\begin{align*}
\widehat{C}_{i} =\left\{ \begin{array}{cc} \dfrac{1}{1-\dfrac{\Sigma_{n}}{\Gamma_{n}}} & \mathrm{if}\ \ i=n\\\\ \dfrac{\Sigma_{i}}{\Gamma_{i}} \dfrac{1}{1-\dfrac{\Sigma_{n}}{\Gamma_{n}}} & \mathrm{if}\ \  i\neq n.
\end{array} \right.
\end{align*}
}
 \caption{Regenerative Ulam-von Neumann Algorithm for the computation of a single column of the matrix inverse}\label{alg:ruvn2}
\end{algorithm}

Algorithm \ref{alg:ruvn} leverages a single Markov chain where each time 
the chain transitions from state $i\in[d]$ to state $j\in[d]$, the algorithm updates the 
$j$-th columns of the matrices $\Gamma$ and $\Sigma$, and sets the $j$-th column of $\Lam$ to 
zero. Let now $n\in[d]$ denote the column index of the matrix inverse $B^{-1}$ that we sought
to compute and consider a variant of Algorithm \ref{alg:ruvn} such that the matrices $\Gamma$, 
$\Sigma$, and $\Lam$, are updated only when $j\equiv n$. Thus, only the $n$-th columns of the 
matrices $\Gamma$, $\Sigma$, and $\Lam$ are modified, with the rest of the $d-1$ columns being identically zero. As we show next, the above variant computes the same approximation 
$\widehat{C}_{i,n}$ as Algorithm \ref{alg:ruvn}.

To verify the above, it suffices to show that the computation of the $n$-th 
column of the matrix $\widehat{C}$ does not exploit any entries outside the $n$-th column of 
the matrices $\Gamma$, $\Sigma$, and $\Lam$. A careful look at the output formula of Algorithm 
\ref{alg:ruvn} reveals that $\widehat{C}_{i,n}$ depends solely on the scalars $\Sigma_{n,n},\ \Gamma_{n,n}$ (when $i\equiv n$) and the scalars $\Sigma_{i,n},\ \Gamma_{i,n},\ \Sigma_{n,n},\ \Gamma_{n,n}$ (when $i\neq n$). 
Thus, the computation of each $\widehat{C}_{i,n}$ involves only the $n$-th column of the matrices $\Sigma$ and $\Gamma$. Moreover, the scalars $\Sigma_{i,n}$ and $\Gamma_{i,n}$ 
involve the matrix $\Lam$ only through the updates $\Gamma_{i,n} = \Gamma_{i,n}+1$ if 
$\Lam_{i,n}\neq0$ and $\Sigma_{i,n} = \Sigma_{i,n}+\Lam_{i,n}$. Therefore, we only need 
to access the $n$-th column of the matrix $\Lambda$ as well.

Algorithm \ref{alg:ruvn2} summarizes the special case where Algorithm 
\ref{alg:ruvn} is used to compute the $n$-th single column of the matrix inverse $B^{-1}$. 
For the sake of clarity, we have chosen to retain the same symbols as in Algorithm 
\ref{alg:ruvn} even though the variables $\Lam,\ \Sigma$, and $\Gamma$ are now vectors. 
In a nutshell, Algorithm \ref{alg:ruvn2} proceeds identically to Algorithm \ref{alg:ruvn} 
except that $\Sigma$ and $\Gamma$ are now updated only when $n\equiv j$, i.e., we still 
sample the next state according to $P_{i,:}$ but we ignore any update of the variables 
$\Sigma$ and $\Gamma$ unless $n\equiv j$. Similarly, $\Lam$ is still modified 
as $\Lam_{i}=\Lam_{i}+\ONE_{[\Lam_{i}=0]}$ and updated $\Lam=\Lam\times \frac{A_{i,j}}{P_{i,j}}$ 
but it is reset to $\Lam=0$ only when $n\equiv j$. Algorithm \ref{alg:ruvn2} requires less 
storage compared to Algorithm \ref{alg:ruvn} since the $d\times d$ matrix variables become 
vectors of length $d$.

\bibliographystyle{siamplain}

\end{document}